\documentclass{article}

\usepackage{arxiv}

\usepackage[utf8]{inputenc} 
\usepackage[T1]{fontenc}    
\usepackage{hyperref}       
\usepackage{url}            
\usepackage{booktabs}       
\usepackage{amsfonts}       
\usepackage{nicefrac}       
\usepackage{microtype}      
\usepackage{lipsum}		
\usepackage{graphicx}
\usepackage{natbib}
\usepackage{doi}

\usepackage{amssymb}
\usepackage{amsmath}
\usepackage{amsthm}
\usepackage[utf8]{inputenc}
\usepackage{mathtools}
\usepackage{mathrsfs}
\usepackage{multirow}
\usepackage{import}
\usepackage{tikz-cd}
\usepackage[symbols,nogroupskip,sort=none]{glossaries-extra}

\newcommand{\keyword}[1]{\textbf{#1}}

\newcommand{\innerprod}[3]{\langle #1,#2 \rangle_{#3}}
\newcommand{\pinnerprod}[3]{\big( #1,#2 \big)_{#3}}

\newcommand{\tangentbd}[1]{\mathrm{T}#1}

\newcommand{\TM}[1]{\mathrm{T}_{#1} M}

\newcommand{\TMC}[1]{\mathrm{T}_{#1}M^{\mathbb{C}}}

\newcommand{\TstarM}[1]{\mathrm{T}_{#1}^{*} M}

\newcommand{\g}{\mathfrak{g}}
\newcommand{\gC}{\mathfrak{g}^{\mathbb{C}}}
\newcommand{\gk}[1]{\mathfrak{g}^{#1}}
\newcommand{\gpq}[2]{\mathfrak{g}^{#1,#2}}

\newcommand{\gstar}{\mathfrak{g}^{*}}
\newcommand{\gstarC}{\mathfrak{g}^{*\mathbb{C}}}

\newcommand{\gstarpq}[2]{{\mathfrak{g}^*}^{#1,#2}}

\newcommand{\secsp}[2]{\Gamma(#1,#2)}
\newcommand{\smoothfuncsp}[1]{C^{\infty}(#1)}

\newcommand{\Ak}[1]{\mathcal{A}^{#1}_{\mathfrak{g}}}
\newcommand{\Apq}[2]{\mathcal{A}^{#1,#2}_{\mathfrak{g}}}
\newcommand{\cohomo}[2]{\mathrm{H}^{#1}_{#2}}

\newcommand{\diff}{\mathrm{d}}
\newcommand{\LieagbdDiff}{\mathrm{d}_{\mathfrak{g}}}
\newcommand{\LiegpDiff}{\mathrm{d}_{G}}
\newcommand{\KoszulDiff}{\mathrm{d}_{\nabla}}
\newcommand{\LieagbdConn}[2]{\prescript{\mathfrak{g}}{}{\nabla^{#1}_{#2}}}
\newcommand{\LieagbdChernConn}[2]{\prescript{\mathfrak{g}}{}{\nabla^{#1}_{#2}}}
\newcommand{\LieagbdLpls}[2]{\prescript{\mathfrak{g}}{}{\Delta^{#1}_{#2}}}

\newcommand{\LieagbdPartial}{\partial_{\mathfrak{g}}}
\newcommand{\partialbar}{\bar{\partial}}
\newcommand{\LieagbdPartialbar}{\bar{\partial}_{\mathfrak{g}}}

\newcommand{\C}{\mathbb{C}}
\newcommand{\R}{\mathbb{R}}

\newcommand{\rank}{\mathrm{rank}}

\newcommand{\kernel}{\mathrm{ker}}

\newcommand{\Sym}{\mathrm{Sym}}
\newcommand{\anchor}{\rho}
\newcommand{\End}{\mathrm{End}}
\newcommand{\Hom}{\mathrm{Hom}}
\newcommand{\Tr}{\mathrm{Tr}}
\newcommand{\Ind}{\mathrm{Ind}}
\newcommand{\Th}{\mathrm{Th}}
\newcommand{\T}{\mathrm{T}}
\newcommand{\chg}[1]{\mathrm{ch}^{#1}}
\newcommand{\cg}[1]{\mathrm{c}^{#1}}
\newcommand{\Tdg}[1]{\mathrm{Td}^{#1}}
\newcommand{\euler}[1]{\mathrm{eu}^{#1}}

\newcommand{\unitarygp}[1]{\mathrm{U}(#1)}
\newcommand{\unitaryLalgb}[1]{\mathfrak{u}(#1)}

\newcommand{\basexi}[1]{\xi_{#1}}
\newcommand{\basev}[1]{v_{#1}}
\newcommand{\basedualxi}[1]{\xi^{'}_{#1}}
\newcommand{\basedualv}[1]{v^{'}_{#1}}

\newcommand{\indexi}[1]{i_{#1}}

\newcommand{\order}[1]{\mathrm{order}\big( #1 \big)}

\newcommand{\sgn}[2]{\mathrm{sgn} \Big(\begin{smallmatrix} #1 \\#2\end{smallmatrix}\Big)}

\newcommand{\leafat}[1]{l^{#1}}
\newcommand{\framebd}[1]{\mathrm{F}#1}
\newcommand{\framebdat}[2]{\mathrm{F}_{#2}#1}
\newcommand{\dualframebd}[1]{\mathrm{F}^{*}#1}

\newcommand{\dualframebdC}[1]{\mathrm{F}^{*}#1^{\mathbb{C}}}

\newcommand{\tgsecsp}[2]{\Gamma_{\tau}(#1,#2)}
\newcommand{\tgsmoothfuncsp}[1]{C^{\infty}_{\tau}(#1)}

\newcommand{\tgdiff}{\mathrm{d}_{\tau}}
\newcommand{\tgAk}[1]{\mathcal{A}^{#1}_{\tau}}

\newcommand{\tgcohomo}[1]{\mathrm{H}^{#1}_{\tau}}

\newcommand{\tgApq}[2]{\mathcal{A}^{#1,#2}_{\tau}}
\newcommand{\tgpartial}{\partial_{\tau}}
\newcommand{\tgpartialbar}{\bar{\partial}_{\tau}}
\newcommand{\tgconn}[2]{\prescript{\tau}{}{\nabla_{#2}#1}}
\newcommand{\tgcurvature}{R}
\newcommand{\tgomega}{\prescript{\tau}{}{\omega}}

\newcommand{\tgcherncharact}[1]{\mathrm{ch}^{\tau}(#1)}
\newcommand{\tgcherncharactBiggp}[1]{\mathrm{ch}^{\tau}\bigg(#1\bigg)}
\newcommand{\tgchernclass}[1]{\mathrm{c}^{\tau}(#1)}
\newcommand{\tgdegreechernclass}[2]{\mathrm{c}^{\tau}_{#2}(#1)}
\newcommand{\tgTdclass}[1]{\mathrm{Td}^{\tau}(#1)}
\newcommand{\tgEulerclass}[1]{\mathrm{eu}^{\tau}(#1)}

\newcommand{\fungroup}[1]{\pi_1(#1)}

\newcommand{\btgbd}[1]{\prescript{b}{}{\mathrm{T}#1}}

\newcommand{\bApq}[2]{\mathcal{A}^{#1,#2}_{b}}
\newcommand{\bAk}[1]{\mathcal{A}^{#1}_{b}}

\newcommand{\bDiff}{\mathrm{d}_{b}}
\newcommand{\bPartial}{\partial_{b}}
\newcommand{\bPartialbar}{\bar{\partial_{b}}}

\newcommand{\bConn}[1]{\prescript{b}{}{\nabla_{#1}}}

\newtheorem{theorem}{Theorem}
\newtheorem*{theorem*}{Theorem}
\newtheorem{definition}{Definition}
\newtheorem{proposition}{Proposition}
\newtheorem*{proposition*}{Proposition}
\newtheorem{lemma}{Lemma}
\newtheorem{corollary}{Corollary}

\glsxtrnewsymbol
    [description={Lie groupoid}]
    {Liegpod}
    {\ensuremath{G\rightrightarrows M}}
    
\glsxtrnewsymbol
    [description={Space of k-tuples of composable arrows}]
    {Gk}
    {\ensuremath{G^{(k)}}}
    
\glsxtrnewsymbol
    [description={Complex Lie algebroid}]
    {Lieagbd}
    {\ensuremath{\g}}

\glsxtrnewsymbol
    [description={Complexified Lie algebroid}]
    {LieagbdC}
    {\ensuremath{\gC}}

\glsxtrnewsymbol
    [description={k-degree exterior algebra of $\g$}]
    {gk}
    {\ensuremath{\g^{k}}}

\glsxtrnewsymbol
    [description={$(p,q)$-degree exterior algebra of $\g$}]
    {gpq}
    {\ensuremath{\gpq{p}{q}}}

\glsxtrnewsymbol
    [description={Almost complex structure}]
    {almostcplxstr}
    {\ensuremath{J}}

\glsxtrnewsymbol
    [description={Nijenhuis tensor}]
    {Njtensor}
    {\ensuremath{\mathcal{N}}}

\glsxtrnewsymbol
    [description={Space of Lie algebroid $k$-forms}]
    {Agk}
    {\ensuremath{\Ak{k}(M)}}

\glsxtrnewsymbol
    [description={Space of Lie algebroid $(p,q)$-forms}]
    {Agpq}
    {\ensuremath{\Apq{p}{q}(M)}}

\glsxtrnewsymbol
    [description={Lie algebroid differential}]
    {gDiff}
    {\ensuremath{\LieagbdDiff}}

\glsxtrnewsymbol
    [description={Lie algebroid holomorphic differential}]
    {gPartial}
    {\ensuremath{\LieagbdPartial}}

\glsxtrnewsymbol
    [description={Lie algebroid anti-holomorphic differential}]
    {gPartialbar}
    {\ensuremath{\LieagbdPartialbar}}

\glsxtrnewsymbol
    [description={Lie algebroid connection}]
    {gConn}
    {\ensuremath{\LieagbdConn{}{}}}

\glsxtrnewsymbol
    [description={Lie algebroid curvature tensor}]
    {gCurv}
    {\ensuremath{R(\LieagbdConn{}{})}}

\glsxtrnewsymbol
    [description={Sign of permutation}]
    {sgn}
    {\ensuremath{\sgn
            {\indexi{1},\cdots,\indexi{n}}
            {\tilde{i}_{1},\cdots,\tilde{i}_{n}}}}

\glsxtrnewsymbol
    [description={Sign of ordered array}]
    {ord}
    {\ensuremath{\order{\indexi{1},\cdots,\indexi{n}}
    }}

\glsxtrnewsymbol
    [description={Unitary local frame of $\gpq{1}{0}$ and $\gpq{0}{1}$}]
    {basevanddualv}
    {\ensuremath{\basev{i} \text{ and } \basedualv{i}}
    }

\glsxtrnewsymbol
    [description={Unitary local frame of $\gstarpq{1}{0}$ and $\gstarpq{0}{1}$}]
    {basexianddualxi}
    {\ensuremath{\basexi{i} \text{ and }\basedualxi{i}}
    }

\glsxtrnewsymbol
    [description={Lie algebroid K\"{a}hler form}]
    {kahlerform}
    {\ensuremath{\omega}
    }

\glsxtrnewsymbol
    [description={Transversal density}]
    {transdensity}
    {\ensuremath{\Omega}
    }
    
\glsxtrnewsymbol
    [description={Hodge star}]
    {Hodgestar}
    {\ensuremath{\star}
    }

\glsxtrnewsymbol
    [description={Wedge product of the K\"{a}hler form}]
    {L}
    {\ensuremath{L}
    }

\glsxtrnewsymbol
    [description={Formal adjoint of the wedge product of K\"{a}hler form}]
    {Lambda}
    {\ensuremath{\Lambda}
    }

\glsxtrnewsymbol
    [description={Contraction operator}]
    {iota}
    {\ensuremath{\iota}
    }

\glsxtrnewsymbol
    [description={$\LieagbdPartialbar$-Laplacian}]
    {partialbarLaplacian}
    {\ensuremath{\LieagbdLpls{}{\LieagbdPartialbar}}
    }

\glsxtrnewsymbol
    [description={Differential complex of a Lie groupoid}]
    {diffcplx}
    {\ensuremath{C^{k}_{\mathrm{diff}}(G,E)}
    }

\glsxtrnewsymbol
    [description={Lie groupoid differential}]
    {Liegpoddiff}
    {\ensuremath{\LiegpDiff}
    }

\glsxtrnewsymbol
    [description={Lie groupoid differential cohomology}]
    {Liegpoddiffcohomo}
    {\ensuremath{\cohomo{\bullet}{\mathrm{diff}}(G,E)}
    }

\glsxtrnewsymbol
    [description={Pullback Lie algebroid along the projection}]
    {pi!g}
    {\ensuremath{\pi^{!}\g}
    }

\glsxtrnewsymbol
    [description={Unitary Lie group}]
    {Un}
    {\ensuremath{\unitarygp{n}}
    }

\glsxtrnewsymbol
    [description={Unitary Lie algebra}]
    {un}
    {\ensuremath{\unitaryLalgb{n}}
    }

\glsxtrnewsymbol
    [description={Ring of $\unitarygp{n}$ adjoint action invariant $\C$-valued polynomials on $\unitaryLalgb{n}$}]
    {invRing}
    {\ensuremath{\C[\unitaryLalgb{n}]^{\unitarygp{n}}}
    }

\glsxtrnewsymbol
    [description={Lie algebroid total Chern class}]
    {cgE}
    {\ensuremath{\cg{\g}(E)}
    }

\glsxtrnewsymbol
    [description={Lie algebroid Chern character}]
    {chgE}
    {\ensuremath{\chg{\g}(E)}
    }

\glsxtrnewsymbol
    [description={Lie algebroid Todd class}]
    {TdgE}
    {\ensuremath{\Tdg{g}(E)}
    }

\glsxtrnewsymbol
    [description={G-translation}]
    {Ug}
    {\ensuremath{U_g}
    }

\glsxtrnewsymbol
    [description={Space of G-invariant differential operators on G}]
    {GinvDiffopOnG}
    {\ensuremath{\mathcal{D}_{\mathrm{inv}}(G;\mathbf{E},\mathbf{F})}
    }

\glsxtrnewsymbol
    [description={Space of restrictions G-invariant differential operator on
$G^{(0)}$}]
    {GinvDiffopOnG0}
    {\ensuremath{\mathcal{U}(\g;E,F)}
    }

\glsxtrnewsymbol
    [description={Principal map}]
    {sigmak}
    {\ensuremath{\sigma_{k}}
    }

\glsxtrnewsymbol
    [description={K-group}]
    {kgp}
    {\ensuremath{K_{0}(\bullet)}
    }

\glsxtrnewsymbol
    [description={Higher index}]
    {higherindex}
    {\ensuremath{\Ind_{v}(D)}
    }

\glsxtrnewsymbol
    [description={van Est map}]
    {vanEstmap}
    {\ensuremath{\Phi_\g}
    }

\glsxtrnewsymbol
    [description={Lie algebroid Thom class}]
    {}
    {\ensuremath{\Th^{\g}(H)}
    }

\glsxtrnewsymbol
    [description={Lie algebroid Euler class}]
    {eug}
    {\ensuremath{\euler{\g}}
    }

\glsxtrnewsymbol
    [description={Local index of an elliptic differential operator}]
    {localindex}
    {\ensuremath{\iota_{D}}
    }

\glsxtrnewsymbol
    [description={Tangential de Rham differential}]
    {tgdiff}
    {\ensuremath{\tgdiff}
    }

\glsxtrnewsymbol
    [description={Tangential anti-holomorphic differential}]
    {tgP}
    {\ensuremath{\tgpartialbar}
    }

\glsxtrnewsymbol
    [description={Space of k-degree E-valued tangential de Rham differential forms}]
    {tgAk}
    {\ensuremath{\tgAk{k}(M,E)}
    }

\glsxtrnewsymbol
    [description={Space of $(p,q)$-type E-valued tangential differential forms}]
    {tgApq}
    {\ensuremath{\tgApq{p}{q}(M,E)}
    }

\glsxtrnewsymbol
    [description={Tangential de Rham cohomology}]
    {tgdohomo}
    {\ensuremath{\tgcohomo{\bullet}(M,E)}
    }

\glsxtrnewsymbol
    [description={Tangential connection}]
    {tgconn}
    {\ensuremath{\tgconn{}{}}
    }

\glsxtrnewsymbol
    [description={Tangential Chern class}]
    {tgchernclass}
    {\ensuremath{\tgchernclass{E}}
    }

\glsxtrnewsymbol
    [description={Tangential Chern character}]
    {tgchernch}
    {\ensuremath{\tgcherncharact{E}}
    }

\glsxtrnewsymbol
    [description={Tangential Todd class}]
    {tgtd}
    {\ensuremath{\tgTdclass{E}}
    }

\glsxtrnewsymbol
    [description={Tangential Euler class}]
    {tgeuler}
    {\ensuremath{\tgEulerclass{E}}
    }

\glsxtrnewsymbol
    [description={Holonomy Lie groupoid}]
    {HF}
    {\ensuremath{H_{\mathscr{F}}}
    }

\glsxtrnewsymbol
    [description={Holonomy Lie algebroid}]
    {gF}
    {\ensuremath{\g_{\mathscr{F}}}
    }

\glsxtrnewsymbol
    [description={Tangential K\"{a}hler form}]
    {tgomega}
    {\ensuremath{\tgomega}
    }

\glsxtrnewsymbol
    [description={Average Euler character}]
    {chiM}
    {\ensuremath{\chi(M)}
    }

\glsxtrnewsymbol
    [description={b-tangent bundle}]
    {btgbd}
    {\ensuremath{\btgbd{M}}
    }

\glsxtrnewsymbol
    [description={b-de Rham differential}]
    {bdeff}
    {\ensuremath{\bDiff}
    }

\glsxtrnewsymbol
    [description={b-holomorphic differential and anti-holomorphic differential}]
    {bPb}
    {\ensuremath{\bPartial\text{ and }\bPartialbar}
    }

\glsxtrnewsymbol
    [description={Space of k-degree b-differential forms}]
    {bAK}
    {\ensuremath{\bAk{k}(M)}
    }

\glsxtrnewsymbol
    [description={Space of $(p,q)$-type b-differential forms}]
    {bApq}
    {\ensuremath{\bApq{p}{q}(M)}
    }

\glsxtrnewsymbol
    [description={b-connection}]
    {bconn}
    {\ensuremath{\bConn{}}
    }

\glsxtrnewsymbol
    [description={$\bPartialbar$-Laplace}]
    {bLaplace}
    {\ensuremath{\Delta_{\bPartialbar}}
    }

\title{Lie groupoid Riemann-Roch-Hirzebruch theorem and applications}

\author{ \href{https://orcid.org/0000-0000-0000-0000}{
        \includegraphics[scale=0.06]{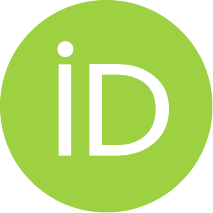}\hspace{1mm}Tengzhou~Hu}
        \thanks{} \\ 
	Department of mathematics\\
	Washington university in St. Louis\\
	Missouri, MO 63130 \\
	\texttt{thu24@wustl.edu} \\
}

\renewcommand{\shorttitle}{\textit{arXiv} Template}

\hypersetup{
pdftitle={A template for the arxiv style},
pdfsubject={q-bio.NC, q-bio.QM},
pdfauthor={David S.~Hippocampus, Elias D.~Striatum},
pdfkeywords={First keyword, Second keyword, More},
}

\begin{document}
\maketitle


\begin{abstract}
    A Lie algebroid is a generalization of Lie algebra that provides a general framework to describe the symmetries of a manifold. In this paper, we introduce Lie algebroid index theory and study the Lie algebroid Dolbeault operator. We also introduce Connes' index theory on regular foliated manifolds to obtain a generalized Riemann–Roch theorem on manifolds with regular foliation. We show that the topological side of Connes' index theory can be identified with the topological side of Lie algebroid index theory. Finally, we introduce Lie algebroid Kodaira vanishing theorem, and provide some applications and examples. The Lie algebroid Kodaira vanishing theorem can be used on the analytic side of Connes' theorem to attest to the criterion of a positive line bundle from its topological information. 
\end{abstract}

\section{Introduction}

Atiyah and Singer first proved the Atiyah–Singer index theorem \cite{AS1971} in 1963. The theorem states that for an elliptic differential operator on a compact manifold, its analytical index, which is the dimension of the kernel minus the dimension of the cokernel, is equal to the topological index, which is defined in terms of the topological invariants of the operator and the topological invariants of the underlying manifold. The theory is motivated by the study of geometry and topology as many problems in geometry often lead to integers that have a topological interpretation, like the Riemann–Roch–Hirzebruch formula and the $\hat{A}$-genus for a spin manifold. Through decades of development, the index theorem has been generalized in a multitude of directions. Connes and Skandalis generalized the index theorem to longitudinal operators on compact foliated manifolds without boundaries at the level of K-theory \cite{CS1984}. Connes also gave a cohomological version of the index theorem for foliations \cite{CO3}.

Pflaum, Posthuma, and Tang further advanced the generalization of the index theorem on Lie algebroids. In \cite{PPT2014}, they gave a cohomological formula for the index of invariant elliptic operators along the fibers of a Lie groupoid over a compact base manifold. The original Atiyah–Singer index theorem and Connes' longitudinal index theorem for foliation can be recovered from it. The topological side of this generalized index theorem is an integration over the dual of the Lie algebroid of cohomology classes of the pullback of the Lie algebroid. Since this is analogous to the cohomological formula of the Atiyah–Singer index theorem as an integral over the cotangent bundle of the original manifold, a natural further step is to generalize the Thom isomorphism for Lie algebroids. The thom isomorphism plays an important role in the classical Atiyah–Singer index theorem and provides a correspondence to the integral formula over the manifold itself. In \cite{PPT2015}, Pflaum, Posthuma, and Tang proved a version of the Thom isomorphism for Lie algebroids. They applied this Thom isomorphism to the index formula in \cite{PPT2014} and obtained an index formula whose topological side is formally the same as the classical topological side except that the characteristic classes take values in Lie algebroid cohomology instead of the de Rham cohomology. We cover their work in detail in Section \ref{sec:G-inv diff op and index thm}.

Building on Pflaum, Posthuma, and Tang's work, we consider their index formula on a complex Lie algebroid and apply it to the Lie algebroid Dolbeault operator. This generates an equality which is similar to the classical Riemann–Roch–Hirzebruch formula.
\begin{theorem*}
Let $\g\to M$ be an integrable complex Lie algebroid and $E\to M$ be a holomorphic vector bundle. If we set $\Phi_{\g}(v)=\Omega_{\g}\in\cohomo{0}{}(\g;L_{\g})$ on the analytic side, then the index formula for the Lie algebroid Dolbeault operator is
\begin{align*}
    \Ind_{\Omega_{\g}}
    (\partialbar+\partialbar^{*})
    =
    \frac{1}{(2\pi\sqrt{-1})^{k}}
    \int_{M}
    (-1)^{\rank(\g)}
    \bigg\langle
    \chg{\g}({\gstar}^{p,0})
    \chg{\g}(E)
    \Tdg{\g}_{\C}(\g)
    ,
    \Omega_{\g}
    \bigg\rangle.
\end{align*}
\end{theorem*}

The analytical index in the above theorem is called the higher index. It is defined by the canonical pairing between the K-theory of convolution algebra and the smooth groupoid cohomology with values in the bundle of transversal densities. Connes' index theorem on foliated manifolds \cite{MS1988} provides a good example of a Lie groupoid determined as the holonomy groupoid of a foliated manifold. Its analytic side is the integration of a measure which is determined by the kernel of the tangential elliptic differential operator, and its topological side agrees with the corresponding part in the Lie algebroid index theorem except for a constant coefficient $(2\pi)^{-k}$ difference, where $k$ is the real dimension of leaves. This builds a bridge connecting the analytic side of Connes' index formula to the topological side of the Lie algebroid index formula. In our paper, we are particular focus on Riemann-Roch formula on Lie algebroid and foliated manifold. The analytic index in Connes' index formula is determined by the kernel of a tangential elliptic operator. Therefore, the kernel of elliptic tangential operators has a canonical relation with the topological property of the base foliated manifold.

We particularly focus on tangential Dolbeault operator on foliated manifolds, which is a Lie algebroid elliptic operator on corresponding holonomy Lie algebroid as well. In {\cite{T2024}, Tengzhou shows the Lie algebroid version of Kodaira vanishing theorem holds on K\"{a}hler Lie algebroid. The theorem states the kernel of the Lie algebroid $\partialbar$-Laplician on Lie algebroid positive line bundle valued $(p,q)$ types vanishes for sufficient large $p+q$. Using this vanishing theorem and the bridge between analytical side and topological side in Riemann-Roch formulas for foliated manifolds and Lie algebroids, we can achieve a criterion of positivity for a tangential holomorphic line bundle over a foliated manifold.
\begin{proposition*}
Let $M$ be a K\"{a}hler foliated manifold with leaves of real dimension $2k$ and $E\to M$ be a tangential line bundle. Then, the line bundle $E$ is not positive if the integration of the characteristic classes
\begin{align*}
    (-1)^{k}
    \tgcherncharact{\dualframebd{M}^{(n,0)}}\,
    \tgcherncharact{E}\,
    \tgTdclass{M}
\end{align*}
is negative.
\end{proposition*}

Our paper is organized as follows. 
\begin{itemize}
    \item In Section \ref{sec:preliminaries}, we firstly introduce the definition of complex Lie agebroids and some basic notations. Then, we introduce Lie algebroid characteristic classes and explain Lie algebroid version of Chern-Weil theory.

    \item In Section \ref{sec:index theory on LIe algebroids}, we first introduce Pflaum, Posthuma, and Tang's work, including the index theorem for equivariant vector bundles and the G-invariant differential operator in \cite{PPT2014}. We then introduce the Lie algebroid Thom class and the simplified version of the index theorem in \cite{PPT2015}. Finally, we focus on the application of the index theorem to the Lie algebroid Dolbeault operator. 

    \item Section \ref{sec:R-R formula for regular foliation} focuses on Connes' index theorem and the Riemann–Roch formula for regular foliation. We introduce the basic concepts necessary for foliation firstly. We then state Connes' index theorem and show its relation with Lie algebroid foliations.
    
    \item Section \ref{sec:app and exp} covers the application of index theory. At the begining, we introduce Lie algebroid Kodaira vanishing theorem. Then, we define a K\"{a}hler and positive line bundle through the terminology of foliation.  The application is focused on manifolds foliated by Riemann surfaces. We then generalize this application to a higher foliation dimension and give a criterion of positivity for a tangential holomorphic line bundle on a K\"{a}hler foliated manifold. An example of foliated K\"{a}hler manifolds is given at the end of this section. 
    
\end{itemize}

\section{Preliminaries}
\label{sec:preliminaries}
\subsection{Lie groupoids and K\"{a}hler Lie algebroids}
A geometric picture of a groupoid consists of two sets: one is the set of objects; the other is the set of arrows (morphisms). A morphism $f$ from $a$ to $b$ is said to have source $a$ and target $b$. In this paper, we always assume the arrow goes from left to right. This picture also explains how the partially defined product works: the product of two arrows is defined if the target of the left arrow is equal to the source of the right arrow, and the product of two such arrows is defined as a new arrow in the groupoid. The Lie groupoid is a groupoid that is also a differentiable manifold. We give its definition as follows.
\begin{definition}
   A \keyword{Lie groupoid} is a groupoid $G\rightrightarrows M$ where $G$ and $M$ are smooth manifolds. It contains the following statements:
    \begin{itemize}
        \item Two surjective submersions $s,t:G\to M$ which are called \keyword{source} and \keyword{target} maps, respectively.
        
        \item A smooth \keyword{multiplication} map $m:\{(g_1,g_2)\in G\times G|s(g_1)=t(g_2)\}\to G$ which will be simply denoted by $g_1\cdot g_2:=m(g_1,g_2)$. A pair such as $g_1,g_2$ is called composable. For a composable pair $g_1,g_2$, the multiplication satisfies $s(g_{1}\cdot g_{2})=s(g_{2})$, $t(g_{1}\cdot g_{2})=t(g_{1})$ and the associativity satisfies $(g_{1}\cdot g_{2})\cdot g_{3}=g_{1}\cdot (g_{2}\cdot g_{3})$.
    
        \item A smooth \keyword{unit map} $u:M\to G$ that satisfies $s\circ u(x)=t\,u(x)$ for any $x\in M$ and  $g\cdot \big(u\circ t(g)\big)=g=\big(u\circ s(g)\big) \cdot g$ for any $g\in G$.
    
        \item A smooth \keyword{inversion} map $i:G\to G$ that satisfies $i(g)\cdot g=u\circ t (g)$ and $g\cdot i(g)=u\circ s (g)$.
    \end{itemize} 
\end{definition}
We use $g^{-1}$ to denote $i(g)$ and $(G\rightrightarrows M,s,t,m,u,i)$ to denote a Lie groupoid. Each element in $G$ is similar to a right arrow going from $s(g)$ to $t(g)$, so naturally the domain of the multiplication map is a set of pairs of arrows where the target of the left arrow agrees with the source of the right arrow. By the assumption that $s$ and $t$ are surjective submersions, the set of composable arrows is a smooth manifold. In general, we consider
\begin{align*}
    G^{(k)}
    :=
    \{
    (g_1,\cdots,g_k)\in G^{\times k}
    |
    s(g_i)=t(g_{i+1}),
    i=1,\cdots,k-1
    \}
\end{align*}
for \keyword{the space of k-tuples of composable arrows}. For $k=0$, we define $G^{(0)}:=u(M)$.

Looking at Lie algebroids, we are mainly interested in the complex structure of a Lie algebroid. Therefore, we mainly follow the definition of a complex Lie algebroid by Weinstein in \cite{W2006} and associate this definition with the notion of almost complex Lie algebroids over almost complex manifolds \cite{IP2016}. The almost complex structure of a Lie algebroid is effectively the natural extension of an almost complex manifold when defining an endomorphism $J$ of the Lie algebroid. This is integrable if its Nijenhuis tensor of $J$ is zero.

\begin{definition}
 A \keyword{complex Lie algebroid} $\g\to M$ is a complex vector bundle over a smooth manifold $M$ with the following structures:
\begin{enumerate}
    \item A \keyword{Lie bracket} $
    [\bullet,\bullet]:
    \secsp{M}{\g}\times\secsp{M}{\g}\to\secsp{M}{\g}
    $, where $\secsp{M}{\g}$ is the space of sections of $\g$. We can naturally extend the Lie bracket to sections of $\g^{\C}$ by linearity.
    
    \item A vector bundle morphism $\rho: \g \to \TMC{}$ which is called an \keyword{anchor map}.
    
    \item \keyword{Leibniz rule} $[v_1,f v_2]=f[v_1,v_2]+ (\rho(v_1) f)v_2$ for all $v_1,v_2 \in \secsp{M}{\g}$ and $f \in \smoothfuncsp{M}$, where $\rho(v_1)f$ is the action of a vector field on a smooth function.
    
    \item The almost complex structure $J:\g \to \g$ from the complex vector bundle structure is integrable. That is saying the Nijenhuis tensor 
    \begin{align*}
        \mathcal{N}(\cdot,\cdot):
        \secsp{M}{\gC}\times\secsp{M}{\gC}
        \to
        \C
    \end{align*}
    such that
    $
        \mathcal{N}(\basev{1},\basev{2}):=
        [J\,\basev{1},J\,\basev{2}]
        -J[J\,\basev{1},\basev{2}]
        -J[\basev{1},J\,\basev{2}]
        -[\basev{1},\basev{2}]
    $ vanishes.
\end{enumerate}
\end{definition}

We use $\g$ to denote a complex Lie algebroid. Lie algebroids in this paper are all complex Lie algebroids unless specified otherwise. Furthermore, we define \textbf{the Lie algebroid} $\g$ \textbf{associated with a Lie groupoid} $G\rightrightarrows M$ to be the vertical bundle along s-fibers on $u(M)$,
\begin{align*}
    \g=\{\ker(s_{*})_{x}\}_{x\in u(M)}\to M,
\end{align*}
where we identify $u(M)\cong M$. The anchor map is the differential of $t$ restricted on the Lie algebroid and the Lie bracket is the natural vector field Lie bracket restricted on $\g$.

Analogous to the decomposition of the complexified tangent bundle $\TMC{}$ on a complex manifold, we decompose the complexified Lie algebroid into 
\begin{align*}
    \gC:=\g\otimes_{\R}\C=\gpq{1}{0}\oplus\gpq{0}{1},
\end{align*}
where $\gpq{1}{0}$ is the eigenspace of $J$ for the eigenvalue $\sqrt{-1}$ and $\gpq{0}{1}$ is the eigenspace of $J$ for the eigenvalue $-\sqrt{-1}$.

Some notations on differential manifolds can be generalized to K\"{a}hler Lie algebroid. We define the notations of the exterior algebra $\g^k:=\bigwedge^{k}\gC$ and $\gpq{p}{q}:=\bigwedge^{p}\gpq{1}{0} \otimes \bigwedge^{q}\gpq{0}{1}$. The natural properties of $\otimes$ and $\oplus$ in the exterior algebra give the equality $\gk{k}=\bigoplus_{p+q=k}\gpq{p}{q}$. A similar construction can be made on the dual of a complex Lie algebroid $\gstar$ as long as this gives a complex structure on $\gstar$ through the pullback of $J$. By using the decompositions $\gC=\gpq{1}{0}\oplus\gpq{0}{1}$ and $\gstarC=\gstarpq{1}{0}\oplus\gstarpq{0}{1}$, we define the conjugate map
\begin{align*}
    \bar{}:\gC  &\to\gC
    \\
    v+\sqrt{-1}J(v)&\mapsto v-\sqrt{-1}J(v)
    \\
    v-\sqrt{-1}J(v)&\mapsto v+\sqrt{-1}J(v)
\end{align*}
for $v\in\g$, and
\begin{align*}
    \bar{}:\gstarC  &\to\gstarC
    \\
    \xi+\sqrt{-1}J(\xi)&\mapsto \xi-\sqrt{-1}J(\xi)
    \\
    \xi-\sqrt{-1}J(\xi)&\mapsto \xi+\sqrt{-1}J(\xi)
\end{align*}
for $\xi\in\gstar$. The definition implies $\bar{v}=v$ for any $v\in\g\subset\gC$ and $\bar{\xi}=\xi$ for any $\xi\in\gstar\subset\gstarC$.

Let $\Ak{k}(M):=\secsp{M}{{\gstar}^{k}}$ be the space of \keyword{Lie algebroid $k$-forms}. Then, the Lie algebroid differential is defined as the generalization of the de Rham differential.
\begin{definition}
The Lie algebroid differential $\LieagbdDiff:\Ak{k}(M) \to \Ak{k+1}(M) $ is defined as
\begin{align*}
    \LieagbdDiff \phi(v_1,\cdots,v_{k+1})
    =
    \begin{pmatrix*}[l]
        \displaystyle\sum_{i=1}^{k+1}
        (-1)^{i}
        \rho(v_i) (\phi(v_1,\cdots,\widehat{v_i},\cdots,v_{k+1}))
        \\
        +
        \\
        \displaystyle\sum_{i<j}
        (-1)^{i+j-1}
        \phi([v_i,v_j],v_{1},\cdots,\widehat{v_i},\cdots,\widehat{v_j},\cdots,v_{k+1})
    \end{pmatrix*}.
\end{align*}
\end{definition}

The definition of $\LieagbdDiff$ is formally same as the definition of the classical de Rham differential, so its application to the exterior product of local frames is formally the same as the case for the classical de Rham. That is

\begin{align*}
    \LieagbdDiff 
    (f\,\zeta_{\indexi{1}}\wedge\cdots\wedge\zeta_{\indexi{k}})
    =
    (\LieagbdDiff f)
    \wedge\zeta_{\indexi{1}}\wedge\cdots\wedge\zeta_{\indexi{k}}
    +
    \displaystyle\sum_{j=1}^{k}
    (-1)^{j}
    f\, \zeta_{\indexi{1}}
    \wedge\cdots\wedge
    (\LieagbdDiff \zeta_{\indexi{j}})
    \wedge\cdots\wedge
    \zeta_{\indexi{k}}
    ,
\end{align*}
where $\{\zeta_{\indexi{}}\}$ is a local frame, and $f$ is a smooth function on $U$. Its definition can also naturally be generalized and applied to the Lie algebroid $(p,q)$-forms $\Apq{p}{q}(M):=\secsp{M}{\gstarpq{p}{q}}$. To achieve the definition of $\LieagbdPartial$ and $\LieagbdPartialbar$ on any degree of $(p,q)$ Lie algebroid differential form, we can check the image of 
\begin{align*}
    \LieagbdDiff:
    \Apq{1}{0}(M)\to\Ak{2}(M)
    =
    \Apq{2}{0}(M)\oplus\Apq{1}{1}(M)\oplus\Apq{0}{2}(M)
\end{align*}
does not actually contain any components in $\Apq{0}{2}(M)$. This is promised since
\begin{align*}
    [\basev{1},\basev{2}]\subset\secsp{M}{\g^{1,0}}
    \qquad\text{and}\qquad
    [\basedualv{1},\basedualv{2}]\subset\secsp{M}{\g^{0,1}}
\end{align*}
for any  $\basev{1},\basev{2}\in\secsp{M}{\g^{1,0}}$ and $\basedualv{1},\basedualv{2}\in\secsp{M}{\g^{0,1}}$ since the almost complex structure J is integrable. Then, using the natural projections
\begin{align*}
    \pi_1:
    \Apq{2}{0}(M)\oplus\Apq{0}{1}(M)
    \to
    \Apq{2}{0}(M)
\qquad\text{and}\qquad
    \pi_2:
    \Apq{2}{0}(M)\oplus\Apq{0}{1}(M)
    \to
    \Apq{0}{1}(M)
    ,
\end{align*}
we define $\LieagbdPartial$ and $\LieagbdPartialbar$ on $\Apq{1}{0}(M)$ by setting
\begin{align*}
    \LieagbdPartial:=
    \pi_{1}\, \LieagbdDiff
    :\Apq{1}{0}(M) \to \Apq{2}{0}(M)
\qquad\text{and}\qquad
    \LieagbdPartialbar:=
    \pi_{2}\, \LieagbdDiff
    :\Apq{1}{0}(M) \to \Apq{1}{1}(M)
    .
\end{align*}
Therefore, we can conclude the decomposition of $\LieagbdDiff$ on $\Apq{1}{0}(M)$:
\begin{align*}
    \LieagbdDiff=
    \LieagbdPartial+\LieagbdPartialbar:
    \Apq{1}{0}(M)
    \to
    \Apq{2}{0}(M)\oplus\Apq{1}{1}(M)
    .
\end{align*}
Analogous to the case of $\Apq{1}{1}(M)$, the similar decomposition of the Lie algebroid differential
\begin{align*}
    \LieagbdDiff=
    \LieagbdPartial+\LieagbdPartialbar:
    \Apq{0}{1}(M)
    \to
    \Apq{1}{1}(M)\oplus\Apq{0}{2}(M)
\end{align*}
also holds on $\Apq{0}{1}(M)$. By checking $
\LieagbdDiff 
    (
        f\,
        \basexi{1}\wedge\cdots\wedge\basexi{p}
        \wedge
        \basedualxi{1}\wedge\cdots\wedge\basedualxi{q}
    )
$, the decomposition $\LieagbdDiff=\LieagbdPartial+\LieagbdPartialbar$ can be generalized to
\begin{align*}
    \LieagbdDiff=
    \LieagbdPartial+\LieagbdPartialbar:
    \Apq{p}{q}(M) \to \Apq{p+1}{q}(M)\oplus\Apq{p}{q+1}(M).
\end{align*}
for any $(p,q)$-degree Lie algebroid differential form.

\begin{definition}
    Suppose $(\g,M,\anchor)$ is a Lie algebroid and $E\to M$ is a vector bundle. A \keyword{Lie algebroid connection} on $E$ is a linear operator
    \begin{align*}
        \LieagbdConn{}{}:\Ak{0}(M,E)\to\Ak{1}(M,E)
    \end{align*}
satisfying the Leibniz rule
\begin{align*}
    \LieagbdConn{}{v}(f\,s)=f\,\LieagbdConn{}{v}(s)+f_{*}\big(\anchor(v)\big)(s),
\end{align*}
where  $v\in\secsp{M}{\g},f\in C^{\infty}(M),s\in\secsp{M}{E}$.
\end{definition}

Analogous to the classical connection construction, we can generalize the Lie algebroid connection to the higher order:
\begin{align*}
    \LieagbdConn{}{}:\Ak{k}(M,E)\to\Ak{k+1}(M,E),
\end{align*}
and define its corresponding \keyword{Lie algebroid curvature tensor} as 
\begin{align*}
    R(\LieagbdConn{}{}):=\LieagbdConn{}{}^2\in\Ak{2}(M,\End(E)).
\end{align*}

In the classical case of integration on a manifold, a non-vanishing section of $\wedge^{top}\TstarM{}$ gives a volume form. However, in the case of a Lie algebroid, any section of $\wedge^{top}\gstar$ does not automatically generate a volume form directly. The solution is reached by introducing \keyword{transversal densities}, which are sections of $L_{\g}:=\bigwedge^{top}\TstarM{} \otimes \bigwedge^{top}\g$.

\begin{definition}
Suppose $\g\to M$ is the Lie algebroid associated with the Lie groupoid $G\rightrightarrows M$. We say the Lie groupoid $G$ is unimodular if there exists a global nonvanishing section $\Omega$ of $L_{\g}$.
\end{definition}
For any $s\in\secsp{M}{\wedge^{top}\gstar}$, its integral over M is 
\begin{align*}
    \int_{M}
    \langle s,\Omega \rangle,
\end{align*}
where $\langle \cdot,\cdot\rangle$ pairs the $\wedge^{top}\g$ and $\wedge^{top}\gstar$ components and produces $\innerprod{s}{\Omega}{}$, a volume form on M.

Recall that the top degree wedge product of a K\"{a}hler form is a canonical volume form. Following this idea, the top degree wedge product of the Lie algebroid K\"{a}hler form $\omega^n\in\secsp{M}{\bigwedge^{top}\gstar}$ gives rise to a Lie algebroid volume form since its combination with the transversal density $\Omega\in \secsp{M}{\bigwedge^{top}\TstarM{} \otimes \bigwedge^{top}\g}$ produces an actual volume form:
\begin{align*}
    \innerprod{\omega^n}{\Omega}{}\in\secsp{M}{\bigwedge^{top}\TstarM{}}.
\end{align*}
We use this canonical Lie algebroid volume form to define the inner product
\begin{align*}
    \pinnerprod{\basexi}{\phi}{\Ak{k}}
    :=
    \int_{M}
    \innerprod
        {\pinnerprod{\basexi}{\phi}{h}\frac{\omega^n}{n!}}
        {\Omega}
        {}
\end{align*}
for $\basexi,\phi \in \Ak{k}(M)$, where $\pinnerprod{}{}{h}$ is the Hermitian metric on $\wedge^{k}\gstar$. The inner product leads to the introduction of an adjoint operator. The formal adjoint of a Lie algebroid differential $\LieagbdDiff$ is an operator
\begin{align*}
    \LieagbdDiff^{*}:\Ak{k+1}(M)\to\Ak{k}(M)
\end{align*}
satisfying $\pinnerprod{\LieagbdDiff^{*}\basexi}{\phi}{\Ak{k}}=\pinnerprod{\basexi}{\LieagbdDiff\phi}{\Ak{k}}$. Similarly, $\LieagbdPartial^{*}$ and $\LieagbdPartialbar^{*}$ are the formal adjoints of $\LieagbdPartial$ and $\LieagbdPartialbar$.

\subsection{Lie algebroid characteristic classes and Chern-Weil theory}

On a manifold, the Chern–Weil theory computes topological invariants of a vector bundle in terms of the Chern class associated with a curvature form. Likewise, we can generalize the Chern–Weil theory to define topological invariants for a vector bundle on a Lie algebroid.

Suppose $(\g,M,\rho)$ is a Lie algebroid with a Lie algebroid connection $\LieagbdConn{}{}$ and $E\to M$ is a vector bundle. We say a Lie algebroid connection admits a \keyword{representation} of $\g$ on E if its curvature $R(\LieagbdConn{}{})=0$. As the curvature tensor can also be written as $R
(\LieagbdConn{}{})(v,w)=[\LieagbdConn{}{v},\LieagbdConn{}{w}]-\LieagbdConn{}{[v,w]}$, we can construct a Lie algebroid cohomology complex
\begin{equation*}
\begin{tikzcd}
    0
    \arrow[r,"\KoszulDiff"]
    &
    \Ak{0}(M,E)
    \arrow[r,"\KoszulDiff"]
    &
    \cdots
    \arrow[r,"\KoszulDiff"]
    &
    \Ak{\rank(\g)}(M,E)
    \arrow[r,"\KoszulDiff"]
    &
    0
\end{tikzcd}
\end{equation*}
such that the differential $\KoszulDiff$ is given by the Koszul formula:
\begin{align*}
    \KoszulDiff\,s(v_0,\cdots,v_k):=
    \begin{pmatrix*}[l]
        \displaystyle\sum_{i=0}^{k}
        (-1)^{i}
        \LieagbdConn{}{v_i}
        s(v_{0},\cdots,\widehat{v_{i}},\cdots,v_{k})
        \\
        +
        \\
        \displaystyle\sum_{i<j}
        (-1)^{i+j-1}
        s([v_i,v_j],v_1,\cdots,\widehat{v_{i}},\cdots,\widehat{v_{j}},\cdots,v_{k})
    \end{pmatrix*}.
\end{align*}
We use $\cohomo{k}{}(\g,E)$ to denote the \keyword{Lie algebroid cohomology} generated by the representation and the complex described above. The cohomology group is simply written as $\cohomo{k}{}(\g)$ when the vector bundle $E$ is trivial. It degenerates to the classical de Rham cohomology when $\g=\TstarM{}$ and $\KoszulDiff=\LieagbdDiff$.

When the vector bundle $E=L_{\g}:=\bigwedge^{top}\TstarM{} \otimes \bigwedge^{top}\g$, we have the following lemma for the integration of Lie algebroid cohomology classes.
\begin{lemma}
Suppose $(\g,M,\rho)$ is a Lie algebroid, then for any $s\in\Ak{\bullet}(M,L_{\g})$ the integral
\begin{align*}
    \int_{M} s
\end{align*}
is well defined and depends only on its cohomology class in $\cohomo{\rank(\g)}{}(\g;L_{\g})$, where we view $s\in\secsp{M}{\wedge^{top}\TstarM{}}$ through the identities 
\begin{align*}
    \Ak{\bullet}(M,L_{\g})
    =
    \secsp{M}{\wedge^{\bullet}\gstar\otimes L_{\g}}
    =
    \secsp{M}{\wedge^{\bullet}\gstar\otimes\wedge^{top}\g\wedge^{top}\TstarM{}}
\end{align*}
and the pairing between $\wedge^{top}\g$ and $\wedge^{top}\gstar$.
\end{lemma}
\begin{proof}
    See reference \cite{ELW1999}.
\end{proof}

Similar to the manifold cases, the complex vector bundle $\g$ with a Hermitian metric carries a structure group $\unitarygp{n}$, which is a unitary matrix group. Then, the Lie algebroid connection $\LieagbdConn{}{}$ can be written into a collection of skew-Hermitian matrices $\{\omega_{\alpha}\in\Ak{1}(U_{\alpha})\otimes\unitaryLalgb{n}\}$ whose entries are Lie algebroid 1-forms, where $\unitaryLalgb{n}$ is the Lie algebra of the unitary group and $\{U_{\alpha}\}$ is a covering of the base manifold $M$. Then, the corresponding Lie algebroid curvature form $\{\Omega_{\alpha}\in\Ak{2}(U_{\alpha},\unitaryLalgb{n})\}$ is a collection of skew-Hermitian matrices whose entries are Lie algebroid 2-forms.

We can extend the classical Chern–Weil theory to define characteristic classes on a Lie algebroid. Let $\C[\unitaryLalgb{n}]^{\unitarygp{n}}$ denote the ring of $\unitarygp{n}$ adjoint action invariant $\C$-valued polynomials on $\unitaryLalgb{n}$, which is the space of $n\times n$ skew-Hermitian matrices. Here, the adjoint invariant means that $h[X]=h[YXY^{-1}]$ for all $X,Y\in\unitaryLalgb{n}$. The ring of k-degree homogeneous polynomial functions on a vector space $V$ is isomorphic to the space of symmetric multilinear functional on $\prod^{k} V$. A polynomial in $\C[\unitaryLalgb{n}]^{\unitarygp{n}}$ can be expressed as 
\begin{align*}
    h[
    \begin{pmatrix}
        x_{1}       &\cdots & x_{n} \\
        \vdots      &       &\vdots \\
        x_{n^2-n}   &\cdots &x_{n^2}\\
    \end{pmatrix}
    ]
    =
    \sum_{i_1,\cdots,i_k=1}^{n^2}
    h_{i_1,\cdots,i_k} 
    x_{i_1}\cdots x_{i_k}
    \in\C[\unitaryLalgb{n}]^{\unitarygp{n}}
    ,
\end{align*}
where the coefficients $h_{i_1,\cdots, i_k}=h_{\sigma(i_1),\cdots ,\sigma(i_k)}$ for any permutation $\sigma$. Such a polynomial can be viewed as a $\unitarygp{n}$ conjugation action invariant k-multilinear functional $\tilde{h}$ on $\prod^k(\unitaryLalgb{n})$ via
\begin{align*}
    \tilde{h}:\prod^k(\unitaryLalgb{n}) 
    &\to
    \C
    \\
    (X^{1},\cdots,X^{k})
    &\mapsto
    \sum_{i_1,\cdots,i_k=1}^{n^2}
    h_{i_1,\cdots,i_k}
    x^{1}_{i_1}\cdots x^{k}_{i_k},
\end{align*}
where the upper index number distinguishes different input matrices
\begin{align*}
    X^{j}=
    \begin{pmatrix}
        x^{j}_{1}       &\cdots &x^{j}_{n}
        \\
        \vdots          &       &\vdots
        \\
        x^{j}_{n^2-n}   &\cdots &x^{j}_{n^2}
    \end{pmatrix}.
\end{align*}

Lie algebroid curvatures are in $\Ak{2}(U_{\alpha})\otimes\unitaryLalgb{n}$, so we need to generalize $\tilde{h}$ to plug the part $\Ak{2}(U_{\alpha})$ into polynomials:
\begin{align*}
    \tilde{h}:\prod^k(\Ak{2}(U_{\alpha})\otimes\unitaryLalgb{n}) 
    &\to
    \Ak{2k}(U_{\alpha})
    \\
    (\phi_{1}\otimes X^{1},\cdots,\phi_{k}\otimes X^{k})
    &\mapsto
    \phi_1\wedge\cdots\wedge\phi_k
    \sum_{i_1,\cdots,i_k=1}^{n^2}
    h_{i_1,\cdots,i_k}
    x^{1}_{i_1}\cdots x^{k}_{i_k}.
\end{align*}

The original $h$ induces an $\Ak{even}(U)$-valued polynomial function via
$
    h(\phi\otimes X)
    =
    \tilde{h}
    (\phi\otimes X, \cdots, \phi\otimes X).
$
We need to be cautious as $\tilde{h}$ is multilinear but $h$ is not. Thus, given a local Lie algebroid curvature 
 $\Omega=\sum\basexi{i}\otimes X^{i}$, we have the following characteristic class:
 \begin{align*}
     h(\Omega)
     &=
        \tilde{h}(\sum\basexi{i}\otimes X^{i},\cdots,\sum\basexi{i}\otimes X^{i})
    \\
    &=
        \sum_{i_1,\cdots,i_k}
        \basexi{i_1}
        \wedge\cdots\wedge
        \basexi{i_k}
        \tilde{h}(X^{i_1},\cdots,X^{i_k})
    .
 \end{align*}

The classical Chern–Weil theory can be generalized to a Lie algebroid. 
\begin{theorem}
    Suppose $\LieagbdConn{}{}$ is a Lie algebroid connection on a complex Lie algebroid $(\g,M,\rho)$ carrying a Hermitian metric and $\{U_{\alpha}\}$, which is an open cover of $M$ by local charts. The complex rank of $\g$ is $n$, while $R(\LieagbdConn{}{})$ is the Lie algebroid curvature corresponding to $\LieagbdConn{}{}$. Given an adjoint action invariant $\C$-valued polynomial $h\in\C[\unitaryLalgb{n}]^{\unitarygp{n}}$, then $h(R(\LieagbdConn{}{}))\in\Ak{even}(M)$ is made up of the collection of even-degree Lie algebroid forms $\{h(\Omega_{\alpha})\in\Ak{even}(U_i)\}$, where $\{\Omega_{\alpha}\in\Ak{2}(U_{\alpha},\unitaryLalgb{n})\}$. It is well defined and closed, i.e., $\LieagbdDiff h(R(\LieagbdConn{}{}))=0$. Moreover, supposing $\tilde{\LieagbdConn{}{}}$ is another Lie algebroid connection, then $h(R(\LieagbdConn{}{}))$ and $h(R(\tilde{\LieagbdConn{}{}}))$ are in a same de Rham cohomology class.
\end{theorem}

We generalize the classical characteristic classes to Lie algebroids.
\begin{itemize}
    \item The Lie algebroid total Chern class $\cg{\g}(E)$ is defined by the polynomial $\mathrm{det}(1+X)$.

    \item The Lie algebroid Chern character $\chg{\g}(E)$ is given by the polynomial $\Tr(e^{X})$, where $e^X$ is the exponential of the matrix.

    \item The Lie algebroid Todd class $\Tdg{g}(E)$ is obtained from $\mathrm{det}(\frac{X}{1-e^{-X}})$.
\end{itemize} 

The classical Chern–Weil theory can be viewed as the ring homomorphism
\begin{align*}
    \C[\unitaryLalgb{n}]^{\unitarygp{n}}
    &\to
    \cohomo{even}{deRham}(M)
    \\
    h
    &\mapsto
    h(R(\nabla)),
\end{align*}
 which is determined by the Hermitian vector bundle $E\to M$ and any connection $\nabla$ on it. This ring homomorphism can also be generalized to the Lie algebroid as the algebraic construction of Lie algebroid characteristic classes is the same as in the classical case. Using $\cohomo{even}{}(\g)$ to replace $\cohomo{even}{deRham}(M)$, we obtain the ring homomorphism
\begin{align*}
    \C[\unitaryLalgb{n}]^{\unitarygp{n}}
    &\to
    \cohomo{even}{}(\g)
    \\
    h
    &\mapsto
    h(E).
\end{align*}
This is convenient: if an equality of polynomials in $\C[\unitaryLalgb{n}]^{\unitarygp{n}}$ holds, then its image through the ring homomorphism also holds, which is an equality of Lie algebroid characteristic classes. Moreover, we can use diagonalization to simplify checking. Since a skew-symmetric matrix is diagonalizable, for any $X=U\Lambda U^{-1}\in\unitaryLalgb{n}$ where $U\in\unitarygp{n}$ and $\Lambda$ is a diagonal matrix composed by eigenvalues of $X$, we can obtain the value $h(X)=h(\Lambda)$ for any adjoint action invariant polynomial $h\in\C[\unitaryLalgb{n}]^{\unitarygp{n}}$. The following two adjoint invariant polynomials are later used in the index theorem:
\begin{align}
    \Tdg{}_{\C}:
    \begin{pmatrix*}
        \lambda_1 & &
        \\
        &\ddots&
        \\
        &&\lambda_n
    \end{pmatrix*}
    \mapsto
    \prod^{n}_{i=1}
    \frac
        {\lambda_i}
        {1-e^{-\lambda_i}}
    \label{EQ-Part 2-(2)-poly for Td}
\end{align}
and
\begin{align}
    e:
    \begin{pmatrix*}
        \lambda_1 & &
        \\
        &\ddots&
        \\
        &&\lambda_n
    \end{pmatrix*}
    \mapsto
    \lambda_{1}\cdots\lambda_{n}.
    \label{EQ-Part 2-(2)-poly for Euler}
\end{align}

%

\section{Index theory on Lie algebroids}
\label{sec:index theory on LIe algebroids}
In this section, we first introduce the index theorem \cite{PPT2015} for the G-invariant differential operators acting on the equivariant vector bundle. The topological side of the index theorem involves integration over the dual of the Lie algebroid of cohomology classes of the pullback of a Lie algebroid. Secondly, we introduce a version of the Thom isomorphism for Lie algebroids and its application on the topological side \cite{PPT2015}. This process is similar to the "integration along the fiber". The integration on the topological side, which is originally over the Lie algebroid, degenerates to the integration over the base manifold. Pflaum, Posthuma, and Tang used this technique to obtain a simplified index formula which is completely similar to the usual cohomological index formula in the Atiyah–Singer index theorem, except that the characteristic classes take values in Lie algebroid cohomology instead of the de Rham cohomology. Notably, the Lie algebroid $\g=\T M$, which is the topological side, is the same as the topological side of the Atiyah index theorem since we have shown that the Lie algebroid characteristic classes degenerate to the classical characteristic classes in this situation. 
\subsection{G-invariant differential operator and index theory}
\label{sec:G-inv diff op and index thm}

Suppose $(G\rightrightarrows M,s,t,m,u,i)$ is a Lie groupoid. Recall that the image of the unit map is $G^{(0)}=u(M)\cong M$, and the Lie algebroid $\g$ associated with $G$ is the vertical bundle with respect to the source map restricted on $G^{(0)}$ with the differential of $s$ restricted on the Lie algebroid as the anchor map. Now, let us introduce some new notations. We say a vector bundle $\mathbf{E}\to G$ is an \keyword{equivariant vector bundle} if it is the pullback bundle $\mathbf{E}:=s^{*}E$ for the vector bundle $E\to M$. Let $\mathbf{F}=s^{*}(F)\to G$ be another equivariant vector bundle. Next, we define the \keyword{G-translation} on sections of the equivariant vector bundle. For any $g\in G$ and $\gamma\in\secsp{G}{\mathbf{E}}$, the G-translation of $\gamma$ about $g$ is a translation
\begin{align*}
    U_g:
    \secsp{t^{-1}(s(g))}{\mathbf{E}|_{t^{-1}(s(g))}}
    \to
    \secsp{t^{-1}(t(g))}{\mathbf{E}|_{t^{-1}(t(g))}}
\end{align*}
such that $(U_g \gamma)(h):=\gamma(hg^{-1})$, where we identify $\gamma(hg^{-1})\in \mathbf{E}_{hg^{-1}}\cong \mathbf{E}_{s(h)}\cong\mathbf{E}_{h}$. The same G-translation $U_g$ for the equivariant vector bundle $\mathbf{F}\to G$ is also defined. Using these two notations, we can define the G-invariant differential operator.
\begin{definition}
A \keyword{G-invariant differential operator} $D$ is a family of classical differential operators 
\begin{align*}
    \{D_x:\secsp{t^{-1}(x)}{\mathbf{E}|_{t^{-1}(x)}}\to\secsp{t^{-1}(x)}{\mathbf{F}|_{t^{-1}(x)}}\}_{x\in M}
\end{align*}
such that all $D_{x}$ obey the G-invariant constraint $D_{t(g)}U_{g}=U_{g}D_{s(g)}$.
\end{definition}
We use $\mathcal{D}_{\mathrm{inv}}(G;\mathbf{E},\mathbf{F})$ to denote the space of G-invariant differential operators on G. Such an operator is entirely determined by its restriction on the germ of $G^{(0)}$. Let us use $\mathcal{U}(\g)$ to denote the space of restrictions of the G-invariant differential operator on $G^{(0)}$. The order of these differential operators naturally gives a filtration $\mathcal{U}(\g)=\cup_{k\geq0}\,\mathcal{U}_{k}(\g)$. In the more general case, the space of restrictions of operators acting between equivariant vector bundles $\mathbf{E}$ and $\mathbf{F}$ is denoted by $\mathcal{U}(\g;E,F)=\cup_{k\geq0}\,\mathcal{U}_{k}(\g;E,F)$, which can naturally be split into $\mathcal{U}(\g;E,F)=\mathcal{U}(\g)\otimes\Hom(E,F)$.

We use $D\in\mathcal{U}(\g;E,F)$ to denote a G-invariant differential operator $D$ as the whole operator on $G$ can be recovered from its restriction on $G^{(0)}$. Since the t-fiber direction on $G^{(0)}$ is the Lie algebroid direction, the principal map of the G-invariant operators descends to the Lie algebroid level:
\begin{align*}
    \sigma_{k}:
    \mathcal{U}_{k}(\g;E,F)
    \to
    \secsp{M}{\Sym^{k}\g\otimes\Hom(E,F)}.
\end{align*}
Then, we say an operator $D\in \mathcal{U}_{k}(\g;E,F)$ is \keyword{elliptic} if its symbol $\sigma_k(D):\secsp{M}{\gstar}\to\Hom(E,F)$ is invertible for all nonzero sections of $\gstar$. This symbol defines an element in the K-group:
\begin{align*}
    \Ind(D)\in K_{0}(\mathcal{A}_G).
\end{align*}
The explanation of the K-group $K_{0}(\mathcal{A}_G)$ and the transformation from a symbol to an element in the K-group can be found in \cite{CO3}. We cannot stop at the K-group $K_{0}(\mathcal{A}_G)$ since the index number should be a real number, at least. In \cite{PPT2015}, the authors introduced a paring between differentiable cohomology classes $\cohomo{\mathrm{ev}}{\mathrm{diff}}(G,L_{\g})$ and an elliptic operator $D\in\mathcal{U}(\g;E,F)$, and defined the higher index:
\begin{align*}
    \Ind_{v}(D):=\langle \chi(v),\Ind(D)\rangle.
\end{align*}
We denote this higher index as the analytical index in index theory.

The following is an index theory for G-invariant differential operators on the Lie groupoid.
\begin{theorem}[\cite{PPT2015}]
\label{thm-Part 2-(3)-unsimplified index thm for Lie gpod}
Let $(G\rightrightarrows M,s,t,m,u,i)$ be a Lie groupoid whose base manifold $M$ is compact and $(\g, M,\rho)$ be its corresponding Lie algebroid. Suppose $E\to M$ and $F\to M$ are vector bundles and the G-invariant differential operator $D\in\mathcal{U}(\g;E,F)$ is elliptic. Let $v$ be a differentiable cohomology class of $G$. Then, the index formula for $D$ is
\begin{align}
    \Ind_{v}(D)
    =
    \frac{1}{(2\pi\sqrt{-1})^{k}}
    \int_{\g^{*}}
    \pi^{*}\Phi_{\g}(v)
    \wedge
    \Tdg{\pi^{!}\g}(\pi^{!}\g\otimes\C)
    \wedge
    \chg{\pi^{!}\g}(\sigma(D)).
\label{EQ-Part 2-(3)-index theorem}
\end{align}
The notations on the right-hand side are described as follows:
\begin{itemize}
    \item $\pi^{!}\g$ is the pullback Lie algebroid along the projection $\pi:\g^{*}\to M$.
    
    \item $\Phi_\g:\cohomo{\bullet}{\mathrm{diff}}(G;L_{\g})\to\cohomo{\bullet}{}(\g;L_{\g})$ is the van Est map for the Lie groupoid G.
    
    \item $\pi^{*}:\cohomo{\bullet}{}(\g;L_{\g})\to\cohomo{\bullet}{}(\pi^{!}\g;\pi^{*}L_{\g})$ is the pullback induced by $\tilde{\pi}:\pi^{!}\g\to\g$ in the construction of the pullback Lie algebroid.

    \item $\Tdg{\pi^{!}\g}(\pi^{!}\g\otimes\C)$ and $\chg{\pi^{!}\g}(\sigma(D))\in\cohomo{\bullet}{}(\pi^{!}\g)$ are Lie algebroid characteristic classes.
\end{itemize}
\end{theorem}


In the Atiyah–Singer index theorem, the topological side is an integration for characteristic classes, where the integration is over the cotangent bundle. To simplify the integration, Thom isomorphism, a group isomorphism between the cohomology group of the total space and the base manifold for a vector bundle, can be used to reduce the integration to the base manifold $M$. A similar technique can be applied to a Lie algebroid. For an oriented vector bundle $\pi:H\to M$ and a Lie algebroid $\g\to M$, the pullback Lie algebroid is $(\pi^{!}\g\to H,\rho_{\pi^{!}\g})$. The corresponding \keyword{Lie algebroid Thom class} is
\begin{align}
    \Th^{\g}(H):=\rho^{*}_{\pi^{!}\g} \Th(H)
\end{align}
via the pullback of the classical Thom class $\Th(H)$ through the anchor map $\rho_{\pi^{!}\g}:\pi^{!}\g\to \T H$. Pflaum, Posthuma and Tang showed that the Lie algebroid Thom class has similar properties to the classical Thom class, and applied the Thom isomorphism on the topological side. The following theorem is the simplified version of Theorem \ref{thm-Part 2-(3)-unsimplified index thm for Lie gpod}.
\begin{theorem}[\cite{PPT2015}]
Suppose $(G\rightrightarrows M,s,t,m,u,i)$ is a Lie groupoid where the base manifold $M$ is compact and $(\g, M,\rho)$ is its corresponding Lie algebroid. Suppose $E\to M$ and $F\to M$ are vector bundles and the G-invariant differential operator $D\in\mathcal{U}(\g;E,F)$ is elliptic. Then, the index formula for $D$ is
\begin{align}
    \Ind_{\beta}(D)
    =
    \frac{1}{(2\pi\sqrt{-1})^{k}}
    \int_{M}
    \frac
        {
        \beta
        \wedge
        \Tdg{\g}(\g\otimes\C)
        \wedge
        \big(\chg{\g}(E)-\chg{\g}(F)\big)
        }
        {\euler{\g}(\gstar)}.
\label{EQ-Part 2-(3)-index theorem on Lie algebroid}
\end{align}
\end{theorem}

\subsection{Application of index theory to the Lie algebroid Dolbeault operator}
For a complex Lie algebroid $(\g, M,\rho)$ and a Lie algebroid holomorphic vector bundle $E\to M$, we can define the \keyword{Lie algebroid Dolbeault operator}
\begin{align*}
    \LieagbdPartialbar+\LieagbdPartialbar^{*}:\Apq{p}{even}(M,E)\to\Apq{p}{odd}(M,E),
\end{align*}
where
\begin{align*}
    \Apq{p}{even}(M,E)
    =
    \bigoplus_{\text{q are even}}
    \Apq{p}{q}(M,E)
    \qquad\text{and}\qquad
    \Apq{p}{odd}(M,E)
    =
    \bigoplus_{\text{q are odd}}
    \Apq{p}{q}(M,E)
    .
\end{align*}
The operator is elliptic, so we can apply the index theorem to it.
\begin{theorem}[The Riemann–Roch formula for Lie algebroids]
\label{thm:The Riemann–Roch formula for Lie algebroids}
In the analytical side, if we set $\Phi_{\g}(v)=\Omega_{\g}\in\cohomo{0}{}(\g;L_{\g})$, then the index formula for the Lie algebroid Dolbeault operator is
\begin{align}
    \Ind_{\Omega_{\g}}
    (\LieagbdPartialbar+\LieagbdPartialbar^{*})
    =
    \frac{1}{(2\pi\sqrt{-1})^{k}}
    \int_{M}
    (-1)^{\rank(\g)}
    \bigg\langle
    \chg{\g}({\gstar}^{p,0})
    \chg{\g}(E)
    \Tdg{\g}_{\C}(\g)
    ,
    \Omega_{\g}
    \bigg\rangle.
    \label{EQ-Part 2-(4)-simplified topo index}
\end{align}
\end{theorem}

\begin{proof}
Looking at the topological side of the (\ref{EQ-Part 2-(3)-index theorem on Lie algebroid}), we use $\Phi_{\g}(v)=\Omega_{\g}$ to replace $\beta$, ${\gstar}^{p,0}\otimes(\bigwedge^{even}{\gstar}^{0,1})\otimes E$ to replace $E$, and ${\gstar}^{p,0}\otimes(\bigwedge^{odd}{\gstar}^{0,1})\otimes E$ to replace $F$. Then, the integration becomes
{\small  
  \setlength{\abovedisplayskip}{6pt}
  \setlength{\belowdisplayskip}{\abovedisplayskip}
  \setlength{\abovedisplayshortskip}{0pt}
  \setlength{\belowdisplayshortskip}{3pt}
\begin{align*}
    \frac{1}{(2\pi\sqrt{-1})^{k}}
    \int_{M}
    \Big\langle
        \frac
            {
            \Tdg{\g}(\g\otimes\C)
            \wedge
            \big(
                \chg{\g}
                \big(
                    {\gstar}^{p,0}\otimes(\bigwedge^{even}{\gstar}^{0,1})\otimes E
                \big)
                -
                \chg{\g}
                \big(
                    {\gstar}^{p,0}\otimes(\bigwedge^{odd}{\gstar}^{0,1})\otimes E
                \big)
            \big)
            }
            {\euler{\g}(\gstar)}
        ,
        \Omega_{\g}
    \Big\rangle 
    ,
\end{align*}
}
where $\innerprod{}{}{}$ annihilates $\wedge^{top} \gstar$ and $\wedge^{top} \g$ in the left and right brackets. The fraction of Lie algebroid characteristic classes inside the integration is almost in the same form as the classical index formula for the Dolbeault operator on a complex manifold. We know that the characteristic class equality 
\begin{align}
    \frac
        {
        \Tdg{}_{\C}(\TM{}\otimes\C)
        \wedge
        \bigg(
        \chg{}(\bigwedge^{even}\TM{})
        -
        \chg{}(\bigwedge^{odd}\TM{})
        \bigg)
        }
        {\euler{}(\TM{})}
    =
    (-1)^{\rank(\TM{})}
    \Tdg{}_{\C}(\TM{})
    \label{EQ-Part 2-(4)-classical char eq}
\end{align}
holds on the manifold; the relevant calculations can be found in \cite{LM1989}. However, this equality cannot be generalized directly to Lie algebroids. Recalling the given polynomial equality $h=0\in\C[\unitaryLalgb{n}]^{\unitarygp{n}}$, we obtain $h(E)=0\in\cohomo{even}{}(\g)$ for any Hermitian vector bundle $E$. To generalize (\ref{EQ-Part 2-(4)-classical char eq}) into Lie algebroid characteristic classes, we need to find an $h=0\in\C[\unitaryLalgb{n}]^{\unitarygp{n}}$ such that $h(\TM{})$ is equal to the right-hand side of (\ref{EQ-Part 2-(4)-classical char eq}). Therefore, we need to:
\begin{itemize}
    \item Rewrite $\Tdg{}_{\C}(\TM{}\otimes\C)$, $\chg{}(\bigwedge^{even}\TM{})$, and $\chg{}(\bigwedge^{odd}\TM{})$ as the characteristic classes of the adjoint invariant polynomials $h_1$, $h_2$, and $h_3$ about $\TM{}$, and construct the polynomial $
    h=
    \frac{h_{1}(h_{2}-h_{3})}{\euler{}}
    -
    (-1)^{\rank(\TM{})}\Tdg{}
    $.
    \item Check that the adjoint invariant polynomial $h$ is equal to $0$ by testing the equality about diagonal matrices.
\end{itemize}

We look at $\chg{}(\bigwedge^{even}\TM{})-\chg{}(\bigwedge^{odd}\TM{})$. Using the splitting principle, we write $\TM{}=l_1\oplus\cdots\oplus l_n$ without loss of generality, where $\{l_i\}$ are real line bundles whose corresponding first Chern classes are $\{x_i\}$. The classical result tells us that
\begin{align*}
    \chg{}(\bigwedge^{even}\TM{})
    -
    \chg{}(\bigwedge^{odd}\TM{})
    =
    \prod^{n}_{i=1}
    (1-e^{x_i}),
\end{align*}
which implies it is the charcteristic class of $\TM{}$ associated to the adjoint invariant polynomial 
\begin{align}
    h:
    \begin{pmatrix*}
    \lambda_1 & &
    \\
    &\ddots&
    \\
    &&\lambda_n
    \end{pmatrix*}
    \mapsto
    \prod^{n}_{i=1}
    (1-e^{\lambda_i})
    .
    \label{EQ-Part 2-(4)-poly for ch of alternative}
\end{align}

Next, we look at $\Tdg{}_{\C}(\TM{}\otimes\C)$. The Todd class $\Tdg{}_{\C}(\TM{}\otimes\C)=\Tdg{}_{\C}(\TM{})\Tdg{}_{\C}(\overline{\TM{}})$ since $\TM{}\otimes\C=\TM{}\oplus\overline{\TM{}}$ and $\Tdg{}_{\C}(E\oplus F)=\Tdg{}_{\C}(E)\Tdg{}_{\C}(F)$. Using the splitting principle again, we assume $\TM{}\cong l_1\oplus\cdots\oplus l_n$, where the first Chern class of $l_i$ is $x_i$. Its conjugate bundle is $\overline{\TM{}} \cong \overline{l_1}\oplus\cdots\oplus \overline{l_n}$, and the first Chern class of $\overline{l_i}$ is $-x_i$. Then, the classical result tells us that
\begin{align*}
    \Tdg{}(\overline{\TM{}})
    =
    \prod^{n}_{i=1}
    \frac
        {-x_i}
        {1-e^{x_i}},
\end{align*}
which implies it is the characteristic class of $\TM{}$ associated to the adjoint invariant polynomial
\begin{align}
    f:
    \begin{pmatrix*}
        \lambda_1 & &
        \\
        &\ddots&
        \\
        &&\lambda_n
    \end{pmatrix*}
    \mapsto
    \prod^{n}_{i=1}
    \frac
        {-\lambda_i}
        {1-e^{\lambda_i}}
    .
    \label{EQ-Part 2-(4)-poly for Td TM bar}
\end{align}
According to (\ref{EQ-Part 2-(4)-poly for ch of alternative}) and (\ref{EQ-Part 2-(4)-poly for Td TM bar}), (\ref{EQ-Part 2-(2)-poly for Td}) is the adjoint invariant polynomial for the Todd class while (\ref{EQ-Part 2-(2)-poly for Euler}) is the adjoint invariant polynomial for the Euler class. We can summarize the adjoint invariant polynomial equality corresponding to (\ref{EQ-Part 2-(4)-classical char eq}). Its left-hand side,
\begin{align*}
    \begin{pmatrix*}
        \lambda_1 & &
        \\
        &\ddots&
        \\
        &&\lambda_n
    \end{pmatrix*}
    \mapsto
    \frac
        {
        \prod^{n}_{i=1}
        \frac
            {\lambda_i}
            {1-e^{-\lambda_i}}
        \prod^{n}_{i=1}
        \frac
            {-\lambda_i}
            {1-e^{\lambda_i}}
        \prod^{n}_{i=1}
        1-e^{\lambda_i}
        }
        {
        \prod^{n}_{i=1}
        \lambda_i
        }
        =
        (-1)^{n}
        \frac
            {\lambda_i}
            {1-e^{-\lambda_i}},
\end{align*}
coincides with the right-hand side.
\end{proof}

\section{Riemann-Roch formula for regular foliation}
\label{sec:R-R formula for regular foliation}
A foliation structure on a foliated manifold naturally induces a Lie groupoid by holonomy equivalence. This Lie groupoid is called the holonomy groupoid, and its corresponding Lie algebroid is the foliation bundle. Connes constructed an index theory on foliated manifolds. By applying the tangential Dolbeault operator to this index theory, we can derive a Riemann-Roch formula on a foliated manifold. We then compare the difference between the two Riemann-Roch formulas obtained from the Lie algebroid index theory and Connes’ index theory.

\subsection{Introduction of foliated manifolds}
\label{sec:Introduction of foliated manifolds}
An n-dimensional manifold $M$ with a k complex dimension foliation structure is the disjoint union of leaves such that each leaf is a k complex dimension manifold. We use $\leafat{x}$ to denote the leaf containing $x$. Then, the \keyword{foliation bundle} $\framebd{M}$ associated to a foliated manifold $M$ is a subbundle of $TM$ such that at any point $x\in M$, $\framebdat{M}{x}=T_x \leafat{x}$, meaning it is the tangent plane of the leaf $\leafat{x}$ at $x$.

Many notations on manifold have corresponding "foliated version" for foliated manifolds. The notations related to de Rham differential are listed as follows:
\begin{itemize}
    \item The \keyword{tangential de Rham differential} $\tgdiff=\{\diff|_{\leafat{}}\}_{\text{all leaves}} :
    \tgsmoothfuncsp{M}
    \to
    \tgsecsp{M}{\dualframebd{M}}$
    is a collection of classical de Rham differentials on each leaf.
    
    \item The space of k-degree vector bundle valued tangential de Rham differential forms is 
    \begin{align*}
        \tgAk{k}(M,E):=\tgsecsp{M}{\wedge^{\bullet}\dualframebd{M}\otimes E},
    \end{align*}
    where $E\to M$ is a tangential vector bundle.

    \item The cohomology obtained via the chain complex of the tangential de Rham differential is the \keyword{tangential de Rham cohomology} $\tgcohomo{\bullet}(M,E)$. 
\end{itemize}

The notations related to complex structure on foliated manifolds are listed as follows:
\begin{itemize}
    \item The collection of integrable almost complex structures on a complex foliated manifold determines the decomposition
    \begin{align*}
        \framebd{M}\otimes\C
        =
        \{\framebd{M}|_{\leafat{}}\otimes\C\}_{\text{all leaves}}
        =
        \{
        \framebd{M}|_{\leafat{}}^{(1,0)}
        \oplus
        \framebd{M}|_{\leafat{}}^{(0,1)}
        \}_{\text{all leaves}},
    \end{align*}
    where $\{\framebd{M}|_{\leafat{}}^{(1,0)}\}=\{{\tangentbd{\leafat{}}}^{(1,0)}\}$ and $\{\framebd{M}|_{\leafat{}}^{(0,1)}\}=\{{\tangentbd{\leafat{}}}^{(0,1)}\}$ are a collection of \textbf{holomorphic} and \textbf{antiholomorphic tangent bundles} of leaves.

    \item The space of all $(p,q)$-type vector bundle valued tangential differential forms is
    \begin{align*}
        \tgApq{p}{q}(M,E):=\tgsecsp{M}{\framebd{M}^{(p,q)}\otimes E}.
    \end{align*}

    \item The \keyword{tangential $\tgpartialbar$-operator} is
    \begin{align*}
        \tgpartialbar:=
        \pi^{(p,q+1)}\, \tgdiff
        :
        \tgApq{p}{q}(M,E)
        \to
        \tgApq{p}{q+1}(M,E)
        ,
    \end{align*}
    which is equal to $\{\partialbar|_{l}\}_{\text{all leaves}}$, a collection of $\partialbar$-operators restricted on each leaf.
\end{itemize}

Differential operators and connections also have their "tangential version". The related notations are listed as follows:
\begin{itemize}
    \item A \keyword{tangential differential operator} $D:\tgsecsp{M}{E_1}\to\tgsecsp{M}{E_2}$ is a continuous linear operator given by a collection of smooth differential operators along the leaves: $D=\{D|_{\leafat{}}\}_{\text{all leaves}}$, where $E_1$ and $E_2$ are two tangential smooth vector bundles over $M$. 

    \item A \textbf{tangential connection }
    \begin{align*}
        \tgconn{}{}:\tgsecsp{M}{E}\to\tgsecsp{M}{\dualframebdC{M}\otimes E}
    \end{align*}
    can be understood as a smooth collection of classical leafwise connections.

    \item The \keyword{tangential curvature} associated with $\tgconn{}{}$ is
    \begin{align*}
        \tgcurvature(\tgconn{}{}):=
        \tgconn{}{}\,\tgconn{}{}:
        \tgsecsp{M}{E}
        \to
        \tgAk{2}(M,E),
    \end{align*}
    whose images are actually tangential smooth $\End(E)$-valued 2-forms. 
\end{itemize}

The characteristic classes for $E\to M$ can also be constructed by plugging the tangential curvature into an invariant polynomial. We define:
\begin{itemize}
    \item the \keyword{tangential Chern class} $\tgchernclass{E}$ using the invariant polynomial $\det(1+X)$,
    \item the \keyword{tangential Chern character} $\tgcherncharact{E}$ using the invariant polynomial $\Tr(e^{X})$,
    \item the \keyword{tangential Todd class} $\tgTdclass{E}$ using the invariant polynomial $\det(\frac{X}{1-e^{-X}})$,
    \item the \keyword{tangential Euler class} $\tgEulerclass{E}$ as the top degree tangential Chern class.
\end{itemize}
The Chern–Weil theory \cite{MS1974} holds for tangential characteristic classes since the tangential characteristic classes restricted on each leaf are just classical characteristic classes.


In Connes' index theorem, the topological index is obtained from the integration of the tangential character classes. The tangential characteristic class, however, does not like the classical case containing a volume form of the base manifold. Recall that a top degree tangential differential form restricted on a leaf is a volume form of this leaf, which induces a measure on the leaf. We call this collection of measures a tangential measure and define it as follows.
\begin{definition}
    A tangential measure $\lambda$ is a collection of usual measures $\{\lambda^{\leafat{}}\}_{\text{all leaves}}$ on each leaf.
\end{definition}

We need a tool to integrate tangential measures, which involves introducing the concepts of transversal and transverse measures. 

\begin{definition}
For a foliated manifold $M$, a transversal is a Borel subset $S$ of $M$ which intersects each leaf in a countable set. It is complete if it meets every leaf.
\end{definition}

The set of transversals forms a $\sigma$-ring. It is closed under the countable unions and relative complementation. A transverse measure is a measure on this $\sigma$-ring that measures the transversal direction of a set where the tangential direction does not attend the result. 

\begin{definition}
A measure is quasi-invariant under the equivalence relation $\mathscr{R}$ if a subset $A\subset X$ is a null set if and only if its saturation $\{x\in X| \exists a\in A\, s.t.\, a\sim x\}$ is a null set. A \keyword{transverse measure} $v$ is a measure on the $\sigma$-ring of Borel transversals $\mathscr{S}$ such that $v_{S}$ is $\sigma$-finite for each $S\in\mathscr{S}$ and quasi-invariant under the holonomy equivalence. It is invariant if it is completely determined by its restriction on any complete transversal.
\end{definition}

Give an invariant transverse measure $v$, we can integrate both tangential measures and top degree tangential cohomology classes. More details can be found in \cite{C1982, MS1988}.


\subsection{Riemann-Roch formula for regular foliation}
In the Atiyah–Singer index theory, the analytical index is obtained from an element in the K theory index group $K_{cpt}(TM)$ which is determined by the principal symbol of the elliptic differential operator. This K group is isomorphic to the integer number group. Thus, the analytical index is an integer number which is equal to the difference between the dimension of the kernel and the dimension of the cokernel. Likewise, the analytical index in Connes' index theorem is in the K theory group $K_{0}(C^{*}_r(G(M)))$, the K-theory group in the context of the $C^{*}$-algebra associated with the pair $(G(M),\lambda)$, where $G(M)$ is the holonomy groupoid and $\lambda$ is the tangential measure. The principal symbol of an elliptic tangential differential operator generates an element in this K-group. Then, taking the partial Chern character from $K_{0}(C^{*}_r(G(M)))$ to $\tgcohomo{\bullet}(M)$, we obtain an element in the tangential cohomology class which determines a tangential measure as the analytical index of the elliptic tangential differential operator. We skip the construction of this $C^{*}$-algebra, K-theory, and the calculation process, whose details can be found in \cite{MS1988}.

The tangential measure determined by an elliptic tangential differential operator is a collection of measures on each leaf. Its restriction on each leaf is determined by the projection operator on the kernel of the restriction of the tangential differential operator on the leaf. We skip the discussion of these details \cite{MS1988} and directly report the result.
\begin{definition}
The local index of an elliptic differential operator $D_l$ on a manifold $\leafat{}$ is
\begin{align*}
    \iota_{D_l}:=\mu_{D_l}-\mu_{D^{*}_l},
\end{align*}
which is a signed Radon measure on $\leafat{}$.
\end{definition}
Let $M$ be a foliated manifold foliated by leaves $\{l\}_{\text{all leaves}}$ and $D=\{D_l\}_{\text{all leaves}}$ be an elliptic tangential differential operator on $M$. On each leaf, $D_{l}$ is an elliptic differential operator and $\iota_{D_l}$ is its local index. The collection of local indices forms a tangential measure which is denoted as 
\begin{align*}
    \iota_{D}:=\{\iota_{D_{l}}\}
    .
\end{align*}

\begin{theorem}[Connes’ index theorem]
Suppose $M$ is a compact foliated manifold with leaves of real dimension $k$ and foliation bundle $E$ which is oriented and equipped with a tangentially smooth oriented tangential Riemannian structure. Let $D$ be a tangentially elliptic differential operator on $E$ and $v$ be an invariant transverse measure. The analytical index $\Ind_v(D)$ and the topological index $\Ind^{topo}_v(D)$ are equal, that is
\begin{align*}
    \Ind_v(D)
    =
    \int 
    \iota_{D} 
    \diff v
\end{align*}
and
\begin{align*}
    \Ind^{topo}_v(D)
    =
    \int 
    \tgcherncharact{D}\tgTdclass{M}
    \diff v
    .
\end{align*}
\end{theorem}
Inside the topological index, $\tgTdclass{M}:=\tgTdclass{\framebd{} M\otimes C}$ is the tangential Todd genus of the manifold M. The Chern characteristic of the tangential elliptic differential operator in the formula is given by
\begin{align*}
    \tgcherncharact{D}
    :=
    (-1)^{\frac{k(k+1)}{2}}
    \phi^{-1}_{\tau}
    \tgcherncharact{\sigma(D)}
    ,
\end{align*}
where $\phi_{\tau}$ is the Thom isomorphism on the foliated manifold and $\tgcherncharact{}$ is the Chern character from the K-group in the tangential cohomology.

Recall that the tangential $\partialbar$-operator on a complex foliated manifold $M$ with a tangential holomorphic vector bundle $E\to M$ induces the tangential Dolbeault operator
 \begin{align*}
     D=\tgpartialbar+\tgpartialbar^{*}:
     \tgApq{p}{even}(M,E)
     \to
     \tgApq{p}{odd}(M,E).
 \end{align*}
We can apply Connes' index theorem to the tangential Dolbeault operator.

\begin{theorem}[The Riemann–Roch formula for regular foliation]
\label{thm:The Riemann–Roch formula for regular foliation}
    Suppose $M$ is a compact complex foliated manifold with leaves of real dimension $2k$ and $E\to M$ is a tangential holomorphic vector bundle which is oriented and equipped with a tangentially smooth oriented tangential Hermitian structure. Connes' index theorem for its tangential Dolbeault operator is stated as follows:
\begin{align}
    \int 
    \iota_{D}
    \diff v
    =
    \int
    (-1)^{k}
    \tgcherncharact{\dualframebd{M}^{(p,0)}}
    \tgcherncharact{E}
    \tgTdclass{M}
    \diff v.
\label{EQ-Foliation and Connes index theorem-index formula for tangential dolbeault op}
\end{align}
\end{theorem}
\begin{proof}
    The Chern character of the tangential Dolbeault operator is:
\begin{align}
    (-1)^{\frac{2k(2k+1)}{2}}
    \phi^{-1}_{\tau}
    \tgcherncharact
        {
        \sigma
        \bigg(
        \tgpartialbar+\tgpartialbar^{*}:
        \tgApq{p}{even}(M,E)
        \to
        \tgApq{p}{odd}(M,E)
        \bigg)
        }
    .
\label{EQ-(2)Foliation and Connes index theorem-the first eq in the proof of index for tg Dolbeault op}
\end{align}
The calculation for the case of the tangential Dolbeault operator is analogous to the classical case since the characteristics are defined by the same invariant polynomial. Thus, we can simplify (\ref{EQ-(2)Foliation and Connes index theorem-the first eq in the proof of index for tg Dolbeault op}) as the process in the classical case:
\begin{align*}
    &
    (-1)^{\frac{2k(2k+1)}{2}}
    \frac
        {
        \tgcherncharactBiggp
            {
            \dualframebd{M}^{(p,0)}
            \otimes
            \big(
            \displaystyle\bigoplus_{\text{q is even}} 
            \dualframebd{M}^{(0,q)}
            \big)
            \otimes E
            }
        -
        \tgcherncharactBiggp
            {
            \dualframebd{M}^{(p,0)}
            \otimes
            \big(
            \displaystyle\bigoplus_{\text{q is odd}} 
            \dualframebd{M}^{(0,q)}
            \big)
            \otimes E
            }
        }
        {\tgEulerclass{M}}
    \\
    =&
    (-1)^{k}
    \tgcherncharact{\dualframebd{M}^{(p,0)}}
    \tgcherncharact{E},
\end{align*}
where $(-1)^{\frac{2k(2k+1)}{2}}=(-1)^{k}$. Multiplying by $\tgTdclass{M}$, we obtain the tangential characteristic classes in (\ref{EQ-Foliation and Connes index theorem-index formula for tangential dolbeault op}).
\end{proof}


\subsection{Difference between Riemann-Roch formula for regular foliation and Lie algebroid}

An important concept for the foliated manifold is the holonomy equivalence. Based on holonomy equivalence, we can define a holonomy Lie groupoid associated to a foliated manifold\cite{W1983}. The Lie algebroid associated with a holonomy Lie groupoid is actually the foliation bundle $\framebd{M}$, whose anchor map is the natural inclusion $\framebd{M}\to\TMC{}$ and the Lie bracket is the natural vector field bracket. Since the natural correspondence between foliation structure and Lie algebroid, we can compare two Riemann-Roch formula canonically.

\begin{proposition}
\label{prop-Connes index theorem-topo side of lie agbd and foliation are same}
    Suppose $M$ is a foliated manifold such that the real dimension of leaves is $k$ and $E$ is a tangential holomorphic vector bundle over $M$. Its foliation structure naturally induces a holonomy Lie groupoid $G(M)$ and the associated Lie algebroid $\g(M)$. The Lie algebroid $\g(M)$ is equal to the foliation bundle $\framebd{M}$; the tangential Dolbeault operator $\tgpartialbar+\tgpartialbar^{*}$ is equal to the Lie algebroid Dolbeault operator $\LieagbdPartialbar+\LieagbdPartialbar^{*}$; and their topological Lie algebroid index and topological Connes' index are equal except for a constant coefficient $(2\pi)^{-k}$ difference.
\end{proposition}

\begin{proof}
The foliated manifold $M$ has a natural corresponding holonomy Lie groupoid $G(M)$ and the associated Lie algebroid $\g(M)$. The associated Lie algebroid $\g(M)$ is actually the frame bundle of $M$ with the natural inclusion as the anchor map:
\begin{equation*}
    \begin{tikzcd}
        \g(M)
        \arrow[r,equal]
        &
        \framebd{M}
        \arrow[r,"\rho"] \arrow[d]
        &
        \TM{}
        \\
        &
        M
        &
    \end{tikzcd}.
\end{equation*}
In this case, the Lie algebroid differential $\LieagbdDiff$ and the tangential differential $\tgdiff$ are the same since $\g(M)=\framebd{M}$.  The space of tangential differential forms and the Lie algebroid differential forms are also equal: 
\begin{align*}
    \Ak{k}(M)
    =\secsp
        {M}
        {\bigwedge^{k} \dualframebd{M}}
    =
    \tgAk{k}(M)
    .
\end{align*}
The split Lie algebroid differential and tangential differential,
\begin{align*}
   \tgpartial+\tgpartialbar
    =
    \tgdiff
    =
    \LieagbdDiff
    =\LieagbdPartial+\LieagbdPartialbar,
\end{align*}
are equal since they are both defined by the complex structure on the foliation bundle. Thus, a Lie algebrod holomorphic vector bundle $E\to M$ is also a tangential holomorphic vector bundle. Moreover, we can obtain the equality of the space of $E$-valued tangential differential forms and $E$-valued Lie algebroid differential forms:
\begin{align*}
    \Apq{p}{q}(M,E)
    =\secsp
        {M}
        {(\bigwedge^{p} \dualframebd{M}^{(1,0)})
        \otimes
        (\bigwedge^{q}\dualframebd{M}^{(0,1)}) \otimes E}
    =
    \tgApq{p}{q}(M,E)
    .
\end{align*}
The Lie algebroid cohomology $\cohomo{p,q}{}(\g,E)$ and the tangential cohomology $\tgcohomo{p,q}(M,E)$ are both given by the same chain complex,
\begin{equation*}
    \begin{tikzcd}
        0
        \arrow[r]
        &
        \tgApq{p}{0}(M,E)
        \arrow[r,"\tgpartialbar"]
        \arrow[d,equal]
        &
        \cdots
        \arrow[r,"\tgpartialbar"]
        &
        \tgApq{p}{top}(M,E)
        \arrow[r,"\tgpartialbar"]
        \arrow[d,equal]
        &
        0
        \\
        0
        \arrow[r]
        &
        \Apq{p}{0}(M,E)
        \arrow[r,"\tgdiff"]
        &
        \cdots
        \arrow[r,"\tgdiff"]
        &
        \Apq{p}{top}(M,E)
        \arrow[r,"\tgdiff"]
        &
        0
    \end{tikzcd},
\end{equation*}
so the cohomology groups $\cohomo{p,q}{}(\g,E)$ and $\tgcohomo{p,q}(M,E)$ are isomorphic.

By comparing Lie algebroid differential forms and tangential differential forms, we also deduce that a tangential Chern connection $\tgconn{}{}$ is a Lie algebroid Chern connection $\LieagbdConn{}{}$. Therefore, given an invariant polynomial $f$, the corresponding tangential characteristic class $f(R(\tgconn{}{}))$ and the Lie algebroid characteristic class $f(R(\LieagbdConn{}{}))$ are equal. We look at the topological Lie algebroid index formula for the Lie algebroid Dolbeault operator,
\begin{align*}
    \frac{1}{(2\pi\sqrt{-1})^{k}}
    \int_{M}
    (-1)^{k}
    \bigg\langle
    \chg{\g}({\gstar}^{p,0})
    \chg{\g}(E)
    \Tdg{\g}_{\C}(\g)
    ,
    \Omega_{\g}
    \bigg\rangle
    ,
\end{align*}
where $\frac{1}{(2\pi\sqrt{-1})^{k}}(-1)^{k}=\frac{1}{(2\pi)^{k}}(-1)^{\frac{k}{2}}$ and the topological Connes' index formula for the tangential Dolbeault operator is
\begin{align*}
    \int
    (-1)^{\frac{k}{2}}
    \tgcherncharact{\dualframebd{M}^{(p,0)}}
    \tgcherncharact{E}
    \tgTdclass{M}
    \diff v
    .
\end{align*}
Using the transversal measure induced by $\Omega_{\g}$, we can conclude that the two index formulas on a foliated manifold are the same if the coefficient $(2\pi)^{-k}$ difference is not considered.
\end{proof}

\section{Applications and examples}
\label{sec:app and exp}

The proposition \ref{prop-Connes index theorem-topo side of lie agbd and foliation are same} builds a bridge connecting the analytical side of the Connes' index formula to the topological side of the Lie algebroid index formula. The analytical side in the Riemann-Roch formula on foliated manifolds is determined by the kernel of the tangential Dolbeault operator. In \cite{T2024}, Tengzhou proved Lie algebroid Kodaira vanishing theorem on K\"{a}hler Lie algebroid. Its corolalry asserts the kernel of the Lie algebroid Dolbeault operator on Lie algebroid positive line bundle valued $(p,q)$ forms vanishes for sufficiently large $p+q$. Since a foliation bundle is naturally a Lie algebroid, the vanishing theorem offers a means to control the size of the kernel, thereby influencing the topological properties of a foliated manifold.

\subsection{Lie algebroid Kodaira vanishing theorem}
Recall that the K\"{a}hler manifold is a complex manifold with a Hermitian metric and the corresponding K\"{a}hler form. The smoothly varying positive-definite Hermitian form on each fiber determined by the Hermitian metric is called the \keyword{canonical $(1,1)$-form}, which is called the K\"{a}hler form if it is closed. For the generalization of the K\"{a}hler manifold to Lie algebroids, let us observe a complex Lie algebroid $\g$ with a Hermitian metric
\begin{align*}
    h:\g^{1,0}\otimes\g^{0,1}\to \C
\end{align*}
which determines a smoothly varying positive-definite Hermitian Lie algebroid form on each fiber. We call this form \keyword{the canonical $(1,1)$-Lie algebroid form} (again, corresponding to the Hermitian metric).

\begin{definition}
    A \keyword{K\"{a}hler Lie algebroid} $(\g,M,\rho)$ is a complex Lie algebroid carrying a Hermitian metric on $\g$ whose associated canonical $(1,1)$-Lie algebroid form $\omega\in\Apq{1}{1}(M)$ is closed. We call $\omega$ a \keyword{K\"{a}hler Lie algebroid form}.
\end{definition}



The notation \keyword{$\LieagbdPartialbar$-Laplacian} is
\begin{align*}
    \LieagbdLpls{}{\LieagbdPartialbar}
    =
    \LieagbdPartialbar^{*}\,\LieagbdPartialbar+\LieagbdPartialbar\,\LieagbdPartialbar^{*}.
\end{align*}
The classical Kodaira vanishing theorem says the kernel of $\partialbar$-Laplacian vanishes on positive line bundle valued $(p,q)$ type for sufficient large $p+q$. To achieve the vanishing result for $\LieagbdPartialbar$-Laplacian, positive line bundles on Lie algebroid are defined as follows.
\begin{definition}
    A Lie algebroid holomorphic line bundle $E\to M$ is \keyword{positive}, if there is a hermitian metric on $E$ such that the curvature $R\in\Apq{1}{1}(M)$ of the Chern connection satisfies that $\sqrt{-1}R$ is positive definite. Meanwhile, we say $E\to M$ is \keyword{negative} if $-\sqrt{-1}R$ is positive definite.
\end{definition}

\begin{theorem}[Lie algebroid Kodaira–Nakano vanishing theorem]
    Give a K\"{a}hler Lie algebroid $\g\to M$ and a Lie algebroid positive line bundle $E\to M$. The kernel of 
    \begin{align*}
        \LieagbdLpls{}{\LieagbdPartialbar}
        :\Apq{p}{q}(M,E)\to\Apq{p}{q}(M,E)
    \end{align*}
    vanishes when $p+q>n$. The same result holds on a negative line bundle under the change of the condition that $p+q<n$.
\end{theorem}

A quick corollary appear since the kernel of $\LieagbdLpls{}{\LieagbdPartialbar}$ and $\LieagbdPartialbar+\LieagbdPartialbar^{*}$ are equal, which will be applied to the analytic index in Connes' index formula.
\begin{corollary}
\label{cor-vanishign theorem for Dolbeault}
Suppose $M\to\g$ is a K\"{a}hler Lie algebroid and $E\to M$ a positive Lie algebroid holomorphic line bundle. The kernel of 
\begin{align*}
    \LieagbdPartialbar+\LieagbdPartialbar^{*}
    :\Apq{p}{q}(M,E)
    \to
    \Apq{p}{q+1}(M,E)
    \oplus
    \Apq{p}{q-1}(M,E)
\end{align*}
vanishes when $p+q>n$. The same result holds on a negative line bundle under the change of the condition that $p+q<n$.
\end{corollary}

\subsection{Application of index theory and the vanishing theorem for foliated K\"{a}hler manifolds}

Let $M$ be a complex foliated manifold. Suppose that each leaf of $M$ is equipped with a K\"{a}hler structure and the leafwise K\"{a}hler structures vary smoothly. That is saying, the collection of leafwise K\"{a}hler form $\{\omega_{l}\}_{\text{all leaves}}$ forms a tangential $(1,1)$-form
\begin{align*}
    \tgomega=\{\omega_{l}\}_{\text{all leaves}}\in\tgApq{1}{1}(M).
\end{align*}
We denote $\tgomega$ as the \keyword{tangential K\"{a}hler form} of $M$. It is easy to see that the tangential  K\"{a}hler form is $\tgdiff$-closed. Conversely, if there is a $\tgdiff$-closed $(1,1)$-form on $M$, then its restriction on each leaf is a classical K\"{a}hler form, making each leaf a K\"{a}hler manifold.
\begin{definition}
A complex foliated manifold $M$ is K\"{a}hler if it has a tangential K\"{a}hler form $\tgomega$. We call such a manifold a K\"{a}hler foliated manifold.
\end{definition}

The foliation bundle $\framebd{M}\to M$ of a K\"{a}hler foliated manifold is a K\"{a}hler Lie algebroid. Under the case of $\g=\framebd{M}$, the terminologies for a Lie algebroid are equivalent to the terminologies for a foliated manifold. For example, $\Apq{p}{q}(M)$ and $\tgApq{p}{q}(M)$ are the same, $\LieagbdDiff$ and $\tgdiff$ are the same, $\LieagbdPartialbar$ and $\tgpartialbar$ are the same. Suppose $E\to M$ is a Hermitian vector bundle and $\LieagbdChernConn{}{}$ is its Lie algebroid Chern connection. Then, $\LieagbdChernConn{}{}$ is also a tangential connection. Moreover, it is compatible with the Hermitian metric of $E$ and its $(0,1)$ part is equal to $\tgpartialbar$. We denote this connection as the \keyword{tangential Chern connection}. Using this terminology, we define a tangential positive line bundle as follows. 
\begin{definition}
A tangential holomorphic line bundle $E\to M$ is \keyword{positive} if there is a Hermitian metric on $E$ such that the tangential curvature $\Theta$ of the corresponding tangential Chern connection satisfying that $\sqrt{-1}\Theta$ is positive definite. Meanwhile, the line bundle is \keyword{negative} if $-\sqrt{-1}\Theta$ is positive definite.
\end{definition}

Let $E\to M$ be negative. Since the operator $\tgpartialbar$ is equal to $\LieagbdPartialbar$ under this case, Corollary \ref{cor-vanishign theorem for Dolbeault} directly gives that the kernel of
\begin{align*}
    (\tgpartialbar+\tgpartialbar^{*}):
    \tgApq{p}{q}(M,E)
    \to
    \tgApq{p}{q}(M,E)
\end{align*}
is trivial for $p+q<n$. In particular, for $(p,q)=(0,0)$, the Dolbeault operator $\tgpartialbar+\tgpartialbar^{*}=\tgpartialbar$ since q is in the lowest degree. The kernel of the tangential Dolbeault operator is then equal to the set of the tangential holomorphic section. Thus, we reach the following result.
\begin{proposition}
There is no $\tgpartialbar$-holomorphic section on the tangential negative line bundle.
\end{proposition}

We now focus on a foliated K\"{a}hler manifold $M$ that carries a positive line bundle $E$ over $M$ and whose leaves are Riemann surfaces. Because of positivity, the kernel
\begin{align*}
    \kernel(\Delta_{\tgpartialbar}:\tgApq{1}{1}(M,E)\to\tgApq{1}{1}(M,E))
    =
    \kernel(\tgpartialbar+\tgpartialbar^{*}:\tgApq{1}{1}(M,E)\to\tgApq{1}{0}(M,E))
\end{align*}
is trivial. We can apply this result to the analytical index part in (\ref{EQ-Foliation and Connes index theorem-index formula for tangential dolbeault op}), the Riemann–Roch formula for singular foliation. For the case $p=1$, the analytical index is
\begin{align}
    \int 
    \iota_{\tgpartialbar+\tgpartialbar^{*}}
    \diff v
    &=
    \int 
    \mu_{\tgpartialbar+\tgpartialbar^{*}
    :\tgApq{1}{0}(M,E)\to\tgApq{1}{1}(M,E)}
    \diff v
    -
    \int
    \mu_{(\tgpartialbar+\tgpartialbar^{*})^{*}
    :\tgApq{1}{1}(M,E)\to\tgApq{1}{0}(M,E)}
    \diff v
\notag
    \\
    &=
    \int
    \mu_{(\tgpartialbar+\tgpartialbar^{*})^{*}
    :\tgApq{1}{0}(M,E)\to\tgApq{1}{1}(M,E)}
    \diff v
\notag
    \\
    &\geq 0
    .
\label{EQ-foliated manifold-analytic index for p=1}
\end{align}

Next, we look at the topological index part in (\ref{EQ-Foliation and Connes index theorem-index formula for tangential dolbeault op}), 
\begin{align*}
    \int
    (-1)^{1}
    \tgcherncharact{\dualframebd{M}^{(p,0)}}
    \tgcherncharact{E}
    \tgTdclass{M}
    \diff v
    .
\end{align*}
The calculation of the tangential characteristic class is analogous to the classical case. The formula of $\tgcherncharact{E}\tgTdclass{M}$ in terms of the first-degree tangential Chern class is
\begin{align*}
    1+\tgdegreechernclass{E}{1}
    +
    \frac{1}{2}\tgdegreechernclass{\framebd{M}}{1}
    .
\end{align*}
Similarly, the formula of $\tgcherncharact{\dualframebd{M}^{(1,0)}}$ in terms of the first-degree tangential Chern class is
\begin{align*}
    1-\tgdegreechernclass{\framebd{M}}{1}.
\end{align*}
Thus, we conclude that the topological index part for $p=1$ is
\begin{align}
    -\int
    \bigg(
    \tgdegreechernclass{E}{1}
    -
    \frac{1}{2}\tgdegreechernclass{\framebd{M}}{1}
    \bigg)
    \diff v
    .
\label{EQ-foliated manifold-topological index for p=1}
\end{align}

The degree-one tangential Chern class $\tgdegreechernclass{\framebd{M}}{1}$ is in top degree since the complex foliation dimension of $M$ is 1. Therefore, $\tgdegreechernclass{\framebd{M}}{1}$ is also the tangential Euler class. We define its integration with respect to a tranversal measure as the \keyword{average Euler character}
\begin{align*}
    \chi(M):=\int \tgEulerclass{\framebd{M}} \diff v.
\end{align*}
A quick result can be obtained by taking $\framebd{M}^{(1,0)}$ as $E$. Plugging $E=\framebd{M}^{(1,0)}$ into (\ref{EQ-foliated manifold-topological index for p=1}), the topological index part for $p=1$ is equal to
\begin{align}
    -\frac{1}{2}\chi.
\label{EQ-foliated manifold-topological index for p=1, using euler character}
\end{align}
Then, associating this to (\ref{EQ-foliated manifold-analytic index for p=1}), the average Euler character is
\begin{align*}
    \chi \leq 0
    .
\end{align*}
This fits the topological picture of the manifold foliated by Riemann surfaces, in which compact leaves have a non-positive Euler character and non-compact leaves do not contribute to the average Euler character.

Recall that the Euler character of a compact Riemann surface of $g$-genus is $2-2g$. Then, the average Euler character is strictly negative when the set of compact leaves of a Riemann surface whose genus is larger than one is not a transversal measure null set. Using (\ref{EQ-foliated manifold-analytic index for p=1}) and (\ref{EQ-foliated manifold-topological index for p=1, using euler character}), we obtain the following proposition.
\begin{proposition}
Let $M$ be a manifold foliated by a Riemann surface. If the set of compact leaves of the Riemann surface whose genus is larger than one is not a transverse measure null set, then the kernel of
\begin{align*}
    \tgpartialbar:\tgApq{1}{0}(M,\framebd^{1,0})\to\tgApq{1}{1}(M,\framebd^{1,0})
\end{align*}
is non-trivial.
\end{proposition}

We now look at the general case and give a criterion of positivity of a tangential holomorphic line bundle on a K\"{a}hler foliated manifold by using topological information of the manifold and the line bundle.
\begin{proposition}
Let $M$ be a K\"{a}hler foliated manifold with leaves of real dimension $2k$ and $E\to M$ be a tangential line bundle. The line bundle $E$ is not positive if the integration of the characteristic classes
\begin{align*}
    (-1)^{k}
    \tgcherncharact{\dualframebd{M}^{(n,0)}}\,
    \tgcherncharact{E}\,
    \tgTdclass{M}
\end{align*}
is negative.
\end{proposition}
\begin{proof}
The condition given in the statement means that the topologcial index
\begin{align*}
    \int
    (-1)^{k}
    \tgcherncharact{\dualframebd{M}^{(n,0)}}
    \tgcherncharact{E}
    \tgTdclass{M}
    \diff v
    <0
\end{align*}
is negative. However, the vanishing theorem states that the analytical index of a tangential Dolbeault operator on a positive line bundle is
\begin{align*}
    \int 
    \iota_{
        \tgpartialbar+\tgpartialbar^{*}:
        \tgApq{n}{even}(M,E)\to\tgApq{n}{odd}(M,E)
        }
    \diff v
    =
    \int
    \mu_{
        (\tgpartialbar+\tgpartialbar^{*})^{*}:
        \tgApq{n}{0}(M,E)\to\tgApq{n}{1}(M,E)
        }
    \diff v
    \geq 0
    .
\end{align*}
We reach a contradiction. Thus, $E$ cannot be positive.
\end{proof}



\subsection{An example of K\"{a}hler foliated manifolds}
Let $X$ be a compact Riemann surface and $\widetilde{X}$ be its universal covering space, which is the upper half complex plane. Let $S$ be a manifold with a homomorphism $\phi:\fungroup{X}\to \mathrm{Diff}(S)$, where $\mathrm{Diff}(S)$ is the space of diffeomorphism on $S$. Then, we construct the manifold
\begin{align*}
    M=\widetilde{X}\times_{\fungroup{X}} S
\end{align*}
as the quotient of $\widetilde{X}\times S$ by identifying $
\big(x,s\big)
\sim
\big(\gamma\cdot x, \phi(\gamma) (s) \big)
$ for any $\gamma\in\fungroup{X}$, where the action of $\fungroup{X}$ on $\widetilde{X}$ is from the deck transformation.

The manifold $M$ is foliated by leaves
$
    \big\{
    \leafat{s}
    =
    \widetilde{X}\times_{\fungroup{X}} \{s\}
    \big\}_{s\in S}
$. The quotient $\widetilde{X}\times_{\fungroup{X}} \{s\}$ is equal to $\widetilde{X}/\mathrm{Stab}\fungroup{X}(s)$, where $\mathrm{Stab}\fungroup{X}(s)$ is the stabilizer of $\fungroup{X}$ at $s$. Thus, all leaves form a collection of K\"{a}hler manifolds whose K\"{a}hler structures are inherited from the upper half complex plane. A leaf is compact when its corresponding stabilizer group is equal to the fundamental group of a compact Riemann surface.

\bibliographystyle{unsrtnat}
\bibliography{references}  






\end{document}